\documentclass[11pt,a4paper]{article}
\usepackage[utf8]{inputenc}
\usepackage{microtype}
\usepackage{amsmath,amsthm,amssymb,amscd}
\usepackage{mathrsfs} 
\usepackage[T1]{fontenc}
\usepackage{ae,aecompl}
\usepackage{times}
\usepackage[british]{babel}
\usepackage{aliascnt}
\usepackage[colorlinks,pagebackref,hypertexnames=false]{hyperref} 
\usepackage{url}
\usepackage[alphabetic,backrefs]{amsrefs}

\usepackage{nth}
\usepackage{float}
\usepackage[all]{xy}
\usepackage{tikz-cd} 
\usepackage{bookmark}
\usepackage[margin=1.9cm]{geometry}
\usepackage{enumerate}
\usepackage{enumitem}

\newcommand{\rd}{{\rm d}}

\newcommand{\rj}{{\rm j}}

\newcommand{\rG}{{\rm G}}

\newcommand{\oep}{{\overline {\Upsilon}}}

\newcommand{\cD}{\mathcal{D}}

\newcommand{\cF}{\mathcal{F}}

\newcommand{\cL}{\mathcal{L}}

\newcommand{\cS}{\mathcal{S}}

\newcommand{\cZ}{\mathcal{Z}}

\newcommand{\fF}{{\mathfrak F}}

\newcommand{\fX}{{\mathfrak X}}

\newcommand{\R}{\mathbb{R}}

\newcommand{\SU}{{\rm SU}}

\renewcommand{\det}{\mathop\mathrm{det}\nolimits}

\renewcommand{\epsilon}{\varepsilon}
\newcommand{\eps}{\epsilon}

\newcommand{\Hol}{\mathrm{Hol}}

\newcommand{\Ric}{{\rm Ric}}

\newcommand{\del}{\partial}

\renewcommand{\Im}{\mathop{\mathrm{Im}}}

\renewcommand{\Re}{\mathop{\mathrm{Re}}}

\newcommand{\vol}{\mathrm{vol}}

\newcommand{\gt}{\texorpdfstring{\mathrm{G}_2}{\space}}

\renewcommand{\div}{\mathrm{div}}

\newcommand{\qandq}{\quad\text{and}\quad}
\newcommand{\qwithq}{\quad\text{with}\quad}
\newcommand{\qforq}{\quad\text{for}\quad}

\def\<{\mathopen{}\left<}
\def\>{\right>\mathclose{}}
\def\({\mathopen{}\left(}
\def\){\right)\mathclose{}}
    
\def\wtilde{\widetilde}

\usepackage{multicol, color}

\definecolor{gold}{rgb}{0.85,.66,0}
\definecolor{cherry}{rgb}{0.9,.1,.2}
\definecolor{burgundy}{rgb}{0.8,.2,.2}
\definecolor{orangered}{rgb}{0.85,.3,0}
\definecolor{orange}{rgb}{0.85,.4,0}
\definecolor{olive}{rgb}{.45,.4,0}
\definecolor{lime}{rgb}{.6,.9,0}
\definecolor{green}{rgb}{.2,.7,0}
\definecolor{grey}{rgb}{.4,.4,.2}
\definecolor{brown}{rgb}{.4,.3,.1}

\newtheorem{theorem}{Theorem}
\newtheorem{prop}[theorem]{Proposition}
\newtheorem{cor}[theorem]{Corollary}

\newtheorem{lemma}[theorem]{Lemma}

\numberwithin{substep}{step}

\numberwithin{subcase}{case}

\theoremstyle{remark}
\newtheorem{remark}[theorem]{Remark}

\theoremstyle{definition}
\newtheorem{definition}[theorem]{Definition}

\usepackage{todonotes}
\usepackage{comment}

\numberwithin{theorem}{section}

\allowdisplaybreaks

\usepackage{titling}
\usepackage{changepage}

\begin{document}
	\title{Laplacian coflows of \texorpdfstring{G\textsubscript{2}}{G2}-structures on contact Calabi--Yau $7$-manifolds}
    \author{Henrique N. Sá Earp, Julieth Saavedra \& Caleb Suan}
    \date{\today}

\maketitle
	
\vspace{-0.5cm}	
\begin{abstract}
We explore three versions of the Laplacian coflow of $\rG_2$-structures on circle fibrations over Calabi--Yau $3$-folds, interpreting their dimensional reductions to the Kähler geometry of the base. Precisely, we reduce Ansätze for the Laplacian coflow, modified or not by {DeTurck's} trick, both on trivial products $CY^3\times S^1$ and on contact Calabi--Yau $7$-manifolds, obtaining in each case a natural modification of the Kähler--Ricci flow.   
\end{abstract}
	
\begin{adjustwidth}{0.95cm}{0.95cm}
    \tableofcontents
\end{adjustwidth}

\newpage

\section{Introduction} \label{Sect:Intro}

\subsection{Context}

We propose an investigation of co-evolving geometric flows, respectively in complex $3$-dimensional Kähler geometry and real $7$-dimensional $\rG_2$-geometry, mediated by dimensional reduction. Similar approaches in various special geometric contexts can be found in a substantial number of recent works, eg. \cites{FineYao2018, HuangWangYao2018,FinoRaffero2020, LambertLotay2021, KennonLotay2023,AshmoreMinasianProto2024}, with specific interests spanning over diverse areas of differential geometry, such as minimal submanifold theory, Yang–Mills theory, and generalized geometry. 
This article extends in a natural way two previous works by the authors and their collaborators, namely \cite{Picard2022flows} and \cite{Lotay2022}, exploring  geometric flows of $\rG_2$-structures on circle fibrations over Calabi--Yau $3$-folds and their repercussions on the Kähler geometry of the base. Concretely, we examine particular Ansätze for the Laplacian coflow, modified by {DeTurck's} trick, on Riemannian products $CY^3\times S^1$, and for Laplacian coflows, modified or not, on contact Calabi--Yau $7$-manifolds. We then interpret their counterparts `downstairs' as modified Kähler--Ricci flows. 

On an oriented and spin $7$-manifold $M$, geometric flows provide a method to deform a $\gt$-structure, given by a non-degenerate form $\varphi\in\Omega^3(M)$, towards `better' structures with special torsion and ultimately metrics with $\gt$ holonomy, which are then Ricci-flat. 
A $\gt$-structure $\varphi$ determines a metric $g_{\varphi}$ and orientation with Riemannian volume form $\vol_{\varphi}$, and its \emph{torsion} $T$ is a 2-tensor which is equivalent to $\nabla^{g_\varphi}\varphi$, see \S \ref{Sect:Prelims.coclosed} below. Pairs $(M^7,\varphi)$ such that $T\equiv0$  are called $\gt$-\emph{manifolds} and are of particular interest, since the holonomy group of $g_{\varphi}$ is then contained in $\gt$. However, complete examples of $\gt$-manifolds are very difficult to construct, especially when $M$ is compact.  
Fernandez and Gray \cite{Fernandez1982} showed that the torsion-free condition is equivalent to $\varphi$ being both \emph{closed} and \emph{coclosed}, i.e,  $\rd\varphi=0$ and $\rd\!\ast_\varphi\!\varphi=0$, where $*_{\varphi}$ is the Hodge star.  This alternative viewpoint on the torsion-free condition as a system of nonlinear PDE is fundamental to several trending methods in $\gt$-geometry.

Our prototypical goal is to study the  \emph{Laplacian coflow (LCF)} of $\gt$-structures, introduced in \cite{Karigiannis2012}:
\begin{align} \label{eq: Laplacian.coflow}
	\frac{\partial\psi_t}{\partial t}
	=\Delta_t\psi_t
	:=(\rd\rd^{\ast_t}+\rd^{\ast_t}\rd)\psi_t,
\end{align}
where $\ast_t$ is the Hodge star of $g_{t}:=g_{\varphi_t}$, the $4$-form  $\psi_t:=\ast_t\varphi_t$ is the dual of the $\gt$-structure,  and $\Delta_t$ 
is the Hodge Laplacian; and the \emph{modified Laplacian coflow (MLCF)} introduced by Grigorian in \cite{Grigorian2013}:
\begin{align} \label{Eq:Modified.coflow}
    \frac{\del}{\del t} \psi_t = \Delta_t \psi_t + \rd\Big( \Big(A - \frac{7}{2}(\tau_0)_t \Big) \varphi_t\Big),
    \qforq A\in\R,
\end{align}
{where $(\tau_0)_t$ is the scalar component of the intrinsic torsion of $\varphi_t$.}

If $M$ is compact, stationary points of \eqref{eq: Laplacian.coflow} would be (dual to) torsion-free $\gt$-structures. Moreover, when an initial condition $\psi_0$ is closed, ie. ~the $\gt$-structure $\varphi_0$ is coclosed, solutions of  \eqref{eq: Laplacian.coflow}  preserve the cohomology class  $[\psi_t]=[\psi_0]\in H^4(M)$, for as long as they exist. Indeed \eqref{eq: Laplacian.coflow} can be interpreted as the gradient flow of \emph{Hitchin's volume functional \cite{Hitchin2001a}} and so the volume of $M$ increases monotonically along the flow, see eg. \cite{Grigorian2013}, however it is not even weakly parabolic; coflows of $\gt$-structures have been  studied eg. by \cites{Karigiannis2012, Grigorian2013, Fino2018, BagagliniFernandezFino2020, Grigorian2020, KennonLotay2023}. On the other hand, the modified Laplacian coflow \eqref{Eq:Modified.coflow} also preserves the coclosed condition and stays within the initial cohomology class, and it does have short-time existence and uniqueness, but the extra term {added to make the flow amenable to DeTurck's trick} introduces stationary points which are not torsion-free. {In particular, if $\varphi$ is a nearly parallel $\gt$-structure, that is
\begin{align}
    d\varphi = \lambda \psi, \quad d\psi = 0, \qforq \lambda > 0
\end{align}
and if $A = \frac{5}{4} \lambda$, then $\varphi$ is a fixed point for the modified coflow.}

In \cite{Picard2022flows}, a thorough analysis is presented on the dynamics of $\rG_2$-flows, in particular relating the Laplacian coflow of $\rG_2$-structures on a trivial product $N^3\times S^1$ of a Calabi--Yau $3$-fold $N$, to Kähler--Ricci flow on the base. On yet another hand, 
\cite{Lotay2022} explores a convenient Ansatz for the LCF of $\rG_2$-structures on contact Calabi--Yau (cCY) 7-manifolds, which are \emph{non-trivial} such products. That investigation unravels the behavior of $\rG_2$-structures under these flows, revealing findings on existence, uniqueness, and the development of singularities. 
It is therefore natural to consider what flows would emerge on the base $3$-folds under the classical modification by a {DeTurck} trick. Thus with this paper we exhaust in total the four cases of Laplacian coflows to consider: {whether or not the circle fibration over the Calabi--Yau 3-fold is trivial, and whether or not the coflow includes Grigorian's modified term. 
In order to adjust expectations, we clarify that, while \cite{Picard2022flows} are able to relate the LCF on a trivial fibration to the well-studied Kähler-Ricci flow on $N^3$, hence obtaining knowledge about the LCF on $M^7$ from the basic flow, we make no similar claim. Rather, we identify flows on $N$ which are indeed \emph{less understood}, since they are to the best of our knowledge completely new in the literature, and may henceforth be conversely \emph{motivated by} the correspondence with the (M)LCF. They could be duly referred to as \emph{modified Kähler-Ricci flows} and hopefully inspire future analytic investigation, which was beyond the scope of this initial study.} 

Adopting a concise review of pertinent literature, we presume the reader's familiarity with $\rG_2$- and Kähler--Ricci flows, aiming to present a short paper where we compute the behavior of $\rG_2$-structures under similar Ansätze for the Laplacian coflow, as well as their induced  modified versions of the Kähler--Ricci flow on the Calabi--Yau $3$-fold. While we will introduce the  immediately necessary concepts and notation, we refer the reader to those two articles and references therein for further background and context. 

\subsection{Overview and main results} \label{Subsect:Overview}

\begin{itemize}
    \item In \S \ref{Sect:Product} we follow \cite{Picard2022flows} and look at solutions to the modified coflow on the product $M^7 = N \times S^1$, where $N$ is a Calabi--Yau $3$-fold. 
    Specifically, in Theorem \ref{Thm:Modified.coflow.product}, we consider a family of $\SU(3)$-structures $(\omega_t, \Upsilon_t)\in (\Omega^{1,1}\times \Omega^{3,0})(N)$ satisfying the system of differential equations
    \begin{align} \label{Eq. system}
    \begin{split}
        \frac{\partial}{\partial t} \omega_t 
        &= - \cL_{\nabla_t (\log|\Upsilon_t|_{\omega_t})} \omega_t + \beta_t , \\
        \frac{\partial}{\partial t} \Upsilon_t 
        &= \cL_{\nabla_t (\log|\Upsilon_t|_{\omega_t})} \Upsilon_t + \gamma_t.
    \end{split}
    \end{align}
    Pulling back the $\SU(3)$-structure to $M$, the family of $\rG_2$-structures  given by $$\varphi_t = \Re \left( \frac{1}{|\Upsilon|_{\omega_t}} \Upsilon \right) + |\Upsilon|_{\omega_t} \rd r \wedge \omega_t$$ 
    is a solution of the modified Laplacian coflow with constant $A$ if, and only if,
    \begin{align*}
        \beta_t \wedge \omega_t 
        &= -A \frac{1}{|\Upsilon|_{\omega}} \rd (\log|\Upsilon_t|_{\omega_t}) \wedge \Re (\Upsilon_t),  \\
        \Im (\gamma_t) 
        &= A|\Upsilon_t|_{\omega_t} \rd (\log|\Upsilon_t|_{\omega_t}) \wedge \omega_t. 
    \end{align*}
    If moreover the complex structure on $N$ is fixed along the flow, then $\beta_t^{(1,1)} = 0$ and $\Im (\gamma_t)^{(0,3) \oplus (3,0)} = 0.$

    \item In \S \ref{Sect:cCY}, we consider families of coclosed $\rG_2$-structures on contact Calabi--Yau (cCY) $7$-manifolds and explore solutions to both the standard and the modified Laplacian coflows. Sasakian deformations that fix the Reeb vector field $\xi$ are characterized by a basic function, in the sense that the contact $1$-form and the transverse Kähler form are given respectively by $$
    \eta_t=\eta + \rd^c f_t
    \qandq
    \omega_t=\omega + \rd \rd^c f_t.
    $$
    Obtaining from this Ansatz the natural family of $\gt$-structures given by $\varphi_t = \Re \left( \frac{1}{|\Upsilon|_{\omega_t}} \Upsilon \right) + |\Upsilon|_{\omega_t} \eta \wedge \omega_t$, we conclude in Theorem \ref{Thm:Laplacian.coflow.cCY.fixed.J} that they are solutions of the Laplacian flow if, and only if,
    \begin{align*}
        \beta_t \wedge \omega_t 
        &=\; 2|\Upsilon_t|^2_{\omega_t} \omega_t^2 - \left( \cL_{\nabla (\log|\Upsilon_t|_{\omega_t})} \rd^c f_t \right) \wedge \Im \Upsilon_t  \\
        & - \Big( \nabla (\log|\Upsilon_t|_{\omega_t}) \lrcorner \omega \Big) \wedge \Im \Upsilon_t + \left( \frac{\del}{\del t} \rd^c f_t \right) \wedge \Im \Upsilon_t, \\
        \Im (\gamma_t) 
        &=\; 4|\Upsilon|_{\omega_t}^2 \rd \left( \log|\Upsilon_t|_{\omega_t} \right) \wedge \omega_t.
    \end{align*}

    Moreover, in Theorem \ref{thm:modified.coflow.cCY.fixed.J}, $\{\varphi_t\}$ will be a solution of the modified Laplacian coflow if, and only if,
    \begin{align*}
        \beta_t \wedge \omega_t &= -|\Upsilon_t|^2_{\omega_t} \omega_t^2 + A |\Upsilon_t|_{\omega_t} \omega_t^2 - \frac{A}{|\Upsilon_t|_{\omega_t}} \rd (\log |\Upsilon_t|_{\omega_t}) \wedge \Re \Upsilon_t \nonumber \\
        &\qquad  - \left( \cL_{\nabla_t (\log|\Upsilon_t|_{\omega_t})} \rd^c f_t \right) \wedge \Im \Upsilon_t - \Big( \nabla_t (\log|\Upsilon_t|_{\omega_t}) \lrcorner \omega \Big) \wedge \Im \Upsilon_t + \left( \frac{\del}{\del t} \rd^c f_t \right) \wedge \Im \Upsilon_t,  \\
        \Im (\gamma_t) &= -2|\Upsilon_t|_{\omega_t}^2 \rd \left( \log|\Upsilon_t|_{\omega_t} \right) \wedge \omega_t+ A |\Upsilon_t|_{\omega_t} \rd \left( \log|\Upsilon_t|_{\omega_t} \right) \wedge \omega_t.
    \end{align*}

    \item In \S \ref{Sect:Modified.coflow.LSES}, we explore solutions to the modified coflow on a contact Calabi--Yau $7$-manifold $(M^7,\eta, \Phi, \Upsilon)$, based on the Ansatz studied in \cite{Lotay2022}:
    $$\varphi_t=b_t^3\Re \Upsilon +a_tb_t^2\eta\wedge\omega, \qquad \varphi_0 = \Re \Upsilon + \epsilon \eta \wedge \omega$$ 
    In particular, Theorem \ref{Thm:modified.coflow.LSES}, Corollary \ref{Cor: modified.coflow.A=0.LSES}, and \S \ref{Subsect:Sing} exhibit the dynamics of such solutions, including singularity formation, for various regimes of the constant $A$ and Sasakian fiber radius $\epsilon$, cf. Table \ref{table:summary}. 

    In the $A = 0$ case, we obtain an explicit expression for the solution of the modified coflow
    $$\varphi_t = (1 - 5\epsilon^2 t)^{\frac{3}{10}} \Re \Upsilon + \epsilon (1 - 5 \epsilon^2 t)^{-\frac{1}{10}} \eta \wedge \omega.$$
    Using this, we also analyze the asymptotic behaviour of the Ansatz solution in the $A = 0$ case near its finite-time singularity.

    \item In \S \ref{Sect:Breaking.Sasakian} we explore solutions on a $cCY^7$ manifold $(M^7, \eta, \Phi, \Upsilon)$ that vary from the initial Sasakian structure by a transverse $\SU(3)$-structure $(\omega'_t, \Upsilon_t)$ and contact $1$-form $\eta_t$, where $\omega'_t \in [\rd \eta_t]_B$:
    \begin{align*}
        \frac{\del}{\del t} \eta_t 
        &= \cL_{\nabla_t (\log |\Upsilon_t|_{\omega'_t})} \eta_t + \alpha_t,  \\
        \frac{\del}{\del t} \omega'_t 
        &= - \cL_{\nabla_t (\log |\Upsilon_t|_{\omega'_t})} \omega'_t + \beta_t, \\
        \frac{\del}{\del t} \Upsilon_t 
        &= \cL_{\nabla_t (\log |\Upsilon_t|_{\omega'_t})} \Upsilon_t + \gamma_t.
    \end{align*}
   Defining a $\gt$-structure in a similar way as before, Theorem \ref{Thm:Laplacian.coflow.ccY.moving} gives a solution to the Laplacian coflow if, and only if, those degrees of freedom satisfy 
    \begin{align*}
        \alpha_t 
        &= 0, \\
        -\eta_t \wedge \gamma_t - \omega'_t \wedge \beta_t 
        &= 2 |\Upsilon|_{\omega'} \rd (\log |\Upsilon|_{\omega'}) \wedge \eta_t \wedge \left[ 3\omega'_t - \rd \eta_t \right] + |\Upsilon|_{\omega'}^2 \rd \eta \wedge \left[ 3\omega'_t - \rd \eta_t \right].
    \end{align*}
    In the case of the modified Laplacian coflow, Theorem \ref{Thm:Modified.coflow.ccY.moving} identifies the constraints
    \begin{align*}
        \alpha_t 
        &= \frac{A}{|\Upsilon_t|_{\omega_t'}} \rd ( \log |\Upsilon_t|_{\omega_t'}), \\
        -\eta_t \wedge \gamma_t - \omega'_t \wedge \beta_t 
        &= 2 |\Upsilon|_{\omega'} \rd (\log |\Upsilon|_{\omega'}) \wedge \eta_t \wedge \left[ 3\omega'_t - \rd \eta_t \right] + |\Upsilon|_{\omega'}^2 \rd \eta \wedge \left[ 3\omega'_t - \rd \eta_t \right] \\
        &\; +A |\Upsilon_t|_{\omega_t'} \rd (\log |\Upsilon_t|_{\omega_t'}) \wedge \eta_t \wedge \omega_t' + A |\Upsilon_t|_{\omega_t'} \rd \eta_t \wedge \omega_t' \\
        &\; -6 |\Upsilon_t|^2_{\omega_t'} \rd (\log |\Upsilon_t|_{\omega_t'}) \wedge \eta_t \wedge \omega_t' -3 |\Upsilon_t|^2_{\omega_t'} \rd \eta_t \wedge \omega_t'.
    \end{align*}
    Finally, a discussion on potentially solving these equations follows in \S \ref{Sect:Possible.solutions}.
\end{itemize}

\bigskip 

\noindent\textbf{Acknowledgements:} The authors would like to thank S\'{e}bastien Picard and Eveline Legendre  for some valuable discussions. The authors also thank the Banff International Research Station (BIRS) and the organizers (Ilka Agricola, Shubham Dwivedi, Sergey Grigorian, Jason Lotay, and Spiro Karigiannis) of the workshop ``Spinoral and octonionic aspects of $\rG_2$ and Spin(7) geometry", where the idea for this paper came about.

CS was supported by a \emph{Four Year Fellowship (4YF) for PhD Students} from the University of British Columbia.
HSE was supported by the  São Paulo Research Foundation (Fapesp)    \mbox{[2021/04065-6]} \emph{BRIDGES collaboration} and the Brazilian National Council for Scientific and Technological Development (CNPq)  \mbox{[311128/2020-3]}.
JPS was supported by Instituto Serrapilheira grant \emph{New perspectives of the min-max theory for the area functional}. 

\subsection{Notation and conventions in  \texorpdfstring{$\rG_2$}{G2}-geometry} \label{Sect:Prelims.coclosed}

Let $(M^7,\varphi)$ be a smooth orientable $\gt$-structure manifold. It determines a Riemannian metric $g_{\varphi}$ and volume $\vol_{\varphi}$ by
$$(X\lrcorner\varphi)\wedge(Y\lrcorner \varphi)\wedge\varphi
=6g_{\varphi}(X,Y)\vol_{\varphi},
\qforq X,Y\in\Gamma(TM),
$$
where $\lrcorner$ denotes the interior product. A $\mathrm{G}_2$-structure gives rise to a $g_\varphi$-orthogonal decomposition of differential forms corresponding to irreducible $\rG_2$-representations: 
\begin{align}
\label{eq:form.decomp}
	\Omega^2 &= \Omega^2_7\oplus\Omega^2_{14}\qandq
	\Omega^3  = \Omega^3_1\oplus\Omega_{7}^{3}\oplus\Omega^3_{27},
\end{align}
where $\Omega^k_l$ has (pointwise) dimension $l$.  Via the Hodge star, this defines isomorphic  decompositions of   $\Omega^5$ and  $\Omega^4$, respectively. Given a $\rG_2$-structure $\varphi$, there exist unique \emph{torsion forms} $\tau_0\in\Omega^0$, $\tau_1\in\Omega^1$,  $\tau_2\in\Omega_{14}^2$ and $\tau_3\in\Omega^3_{27}$, such that
\begin{align}
    \rd\varphi 
    &= \tau_0\psi+3\tau_1\wedge\varphi+\ast\tau_3,
	\label{Eq:Fernandez d} \\
	\rd\psi 
	&= 4\tau_1\wedge\psi+\tau_2\wedge\varphi,
	\label{eq: Fernandez dpsi}  
\end{align}
see e.g.~\cite{Bryant2006}*{Proposition 1}. The \emph{intrinsic torsion} is defined with respect to the Levi-Civita connection of the $\rG_2$-metric by $\nabla\varphi:=\nabla^{g_\varphi}\varphi$. Then, the \emph{full torsion tensor} of $\varphi$ is the $2$-tensor $T$ defined by  
\begin{align}
    \nabla_i\varphi_{jkl}
    =T_i^{m}\psi_{mjkl}, \quad T_i^{\,\, j}
    =\frac{1}{24}\nabla_i\varphi_{lmn}\psi^{jlmn},
\end{align}
see \cite{Karigiannis2007}, and $T_{ij}
=T(\partial_i,\partial_j)$ and $T_i^{\,\, j}=T_{ik}g^{jk}$ and may be expressed in terms of the torsion forms by 
\begin{equation} \label{Eq:Torsion}
	T =\frac{\tau_0}{4}g -\tau_1^{\sharp}\lrcorner\varphi -\frac{1}{2}\tau_2 -\frac{1}{4}\rj_{\varphi}(\tau_3),
\end{equation}
where $j_{\varphi}$ is a linear operator $ \rj_{\varphi}:\Omega^3\rightarrow S^2$ by
\begin{align} \label{Eq:j.operator}
	\rj_{\varphi}(\gamma)(X,Y)
	&=\ast_{\varphi}((X\lrcorner \varphi)\wedge (Y\lrcorner \varphi)\wedge\gamma),
\end{align}
see e.g.~\cite{Karigiannis2007}*{Theorem 2.27}.

\section{The modified Laplacian coflow on \texorpdfstring{$M^7 = N^3 \times S^1$}{M7 =N X S1}} \label{Sect:Product}

We apply the methods from \cite{Picard2022flows} to the modified Laplacian coflow \cite{Grigorian2013}. We note that the sign convention and orientation here are opposite to those chosen in \cite{Picard2022flows}.

Let $M^7 = N^3 \times S^1$, where $N$ is a smooth compact Calabi--Yau $3$-manifold. Let $\omega$ be a K\"{a}hler metric and $\Upsilon$ be a nowhere-vanishing holomorphic $(3,0)$-form on $N$. Both $\omega$ and $\Upsilon$ are closed and, in local Darboux coordinates, we may write
\begin{align}
    \omega &= \frac{i}{2} (g_6)_{p\overline{q}} \rd z^p \wedge \rd \overline{z}^q \\
    \Upsilon &= u \rd z^1 \wedge \rd z^2 \wedge \rd z^3
\end{align}
where $g_6 = (g_6)_{p\overline{q}}$ is the metric associated to $\omega$ and $u$ is a local holomorphic function. The norm of $\Upsilon$ with respect to $\omega$ is given by  
\begin{align}
    |\Upsilon|^2_{\omega} = \frac{|u|^2}{\det (g_6)_{p\overline{q}}}
\end{align}
and it is constant when $\omega$ is Ricci-flat. The pair $(\omega,\Upsilon)$ satisfies the following relations: 
\begin{equation}
    \frac{\omega^3}{3!} = \vol_6 = \frac{i}{8}\frac{1}{|\Upsilon|_{\omega}^2} \Upsilon \wedge \bar{\Upsilon} = \frac{1}{4} \Re \left( \frac{1}{|\Upsilon|_{\omega}} \Upsilon \right) \wedge \Im \left( \frac{1}{|\Upsilon|_{\omega}} \Upsilon \right), 
\end{equation}
where $\vol_6$ is the volume form on $(N,g_6)$. 
The Hodge star operator $\ast_6$ on $N$ has the foowing properties: 
\begin{equation}
    (\ast_6)^2 \alpha 
    = (-1)^k \alpha, \quad \ast_6 \Re (\Upsilon) 
    = \Im \Upsilon, \quad \ast_6 \omega 
    = \frac{1}{2} \omega^2, 
    \qforq
    \alpha\in\Omega^k(N).
\end{equation}
Let $r$ denote the angle coordinate on $S^1$, so $\rd r \in\Omega^1(S^1)$ is the globally defined (volume) form on $S^1$ with respect to its standard round metric.

Now, we consider the natural $\rG_2$-structure on $M$ given by the positive $3$-form
\begin{equation} \label{Eq:varphi.product}
    \varphi = \Re \left( \frac{1}{|\Upsilon|_{\omega}} \Upsilon \right) + |\Upsilon|_{\omega} \rd r \wedge \omega,
\end{equation}
cf \cite{Karigiannis2012}. The $3$-form \eqref{Eq:varphi.product}  induces the metric $g_7$, volume form and dual $4$-form $\psi = \ast \varphi$ given by 
\begin{align}
   g &= |\Upsilon|_{\omega} \rd r^2 + g_6, \label{Eq:g.product}\\
   \vol &= |\Upsilon|_{\omega} \rd r \wedge \vol_6, \label{Eq:vol.product}\\
   \psi &= - \rd r \wedge \Im \Upsilon + \frac{1}{2}\omega^2. \label{Eq:psi.product}
\end{align}
The $7$-dimensional Hodge star operator $\ast$ of $g_\varphi$ has the following properties acting on $\alpha \in \Omega^k(N)$:
\begin{align}
    \ast \alpha &= (-1)^k |\Upsilon|_{\omega} \rd r \wedge \ast_6 \alpha,\\
    \ast (\rd r \wedge \alpha) &= \frac{1}{|\Upsilon|_{\omega}} \ast_6 \alpha. 
\end{align}

Since both $\omega$ and $\Upsilon$ are closed, $\varphi$ is a coclosed $\rG_2$-structure. Moreover, Picard--Suan \cite{Picard2022flows} compute 
\begin{align}
    \rd \varphi &= -\frac{1}{|\Upsilon|_{\omega}} \rd (\log|\Upsilon|_{\omega}) \wedge \Re(\Upsilon)+|\Upsilon|_{\omega} \rd (\log|\Upsilon|_{\omega}) \wedge \rd r \wedge \omega, \label{Eq:d.varphi.product}\\
    \ast \rd \varphi &= \left( \nabla_{g_6} (\log|\Upsilon|_{\omega}) \right) \lrcorner \left( - \rd r \wedge \Im \Upsilon - \frac{1}{2} \omega^2 \right).\label{Eq:star.d.varphi.product}
\end{align}
From those formulae, one easily derives:
\begin{lemma}[\cite{Picard2022flows} {Lemmas 4.5 and 4.6}] \label{Lem:Torsion.product}
\label{Lem:Laplacian.psi.product}
    Let $\varphi$ be the $\rG_2$-structure defined by \eqref{Eq:varphi.product} on $M = N \times S^1$ with $N$ a Calabi--Yau $3$-manifold, then the torsion forms are given by
\begin{equation}
    \tau_0 = 0, \quad \tau_1 = 0, \quad \tau_2 = 0, \quad \tau_3 = \left( \nabla_{g_6} (\log |\Upsilon|_{\omega}) \right) \lrcorner \left( - \rd r \wedge \Im \Upsilon - \frac{1}{2} \omega^2 \right).
\end{equation}
    The Hodge Laplacian of the $4$-form is given by
\begin{equation}
    \Delta \psi=\cL_{\nabla (\log|\Upsilon|_{\omega})} \left( - \rd r \wedge \Im \Upsilon - \frac{1}{2} \omega^2 \right).
\end{equation}
\end{lemma}

Applying Lemma \ref{Lem:Laplacian.psi.product} to the Laplacian coflow \eqref{eq: Laplacian.coflow} of the above $\rG_2$-structures yields the evolution equation
\begin{align}
    \frac{\partial}{\partial t} \left( - \rd r \wedge \Im \Upsilon + \frac{1}{2} \omega^2 \right) = \cL_{\nabla(\log|\Upsilon|_{\omega})} \left( - \rd r \wedge \Im \Upsilon - \frac{1}{2} \omega^2 \right). \nonumber
\end{align}
The terms involving $\omega$ and $\Upsilon$ can be considered separately. Noting time dependencies, one can consider Ans\"{a}tze of the form $(\omega_t,\Upsilon_t)$ on $N$, satisfiying  
\begin{align}
    \frac{\partial}{\partial t} \omega_t &= -\cL_{\nabla_t (\log|\Upsilon_t|_{\omega_t})} \omega_t, \\
    \frac{\partial}{\partial t} \Upsilon_t &= \cL_{\nabla_t (\log|\Upsilon_t|_{\omega_t})} \Upsilon_t. 
\end{align} 
On the other hand, using properties of K\"{a}hler manifolds, we see that
\begin{align} \label{Eq.Ricci.Lie}
    \cL_{\nabla(\log|\Upsilon|_{\omega})}\omega = 2i\partial\overline{\partial}(\log|\Upsilon|_{\omega})=\Ric(\omega,J),
\end{align}
which ultimately relates the Laplacian coflow to the K\"{a}hler--Ricci flow.

\begin{remark} \label{Rem:Compatibility}
    A priori, the structures $(\omega_t,\Upsilon_t)$ along the flow may not remain compatible and integrable for all time. However, the solution presented in  \cite{Picard2022flows} satisfy the required compatibility conditions, as they are obtained by pulling back compatible structures via diffeomorphisms. 
\end{remark}

We now consider the modified coflow in this setting. A similar treatment using ideas from \cite{Picard2020} and \cite{Picard2022flows} {by writing the flows on the base as a modified K\"{a}hler--Ricci flow} yields the following result.

\begin{theorem} \label{Thm:Modified.coflow.product}
    Let $N^3$ be a Calabi--Yau $3$-manifold with K\"{a}hler form $\omega$ and holomorphic $(3,0)$-form $\Upsilon$. Suppose we have a family of compatible $\SU(3)$-structures $(\omega_t,\Upsilon_t)$ satisfying the coupled differential equations
\begin{align}
    \frac{\partial}{\partial t} \omega_t 
    &= -\cL_{\nabla_t (\log|\Upsilon_t|_{\omega_t})} \omega_t + \beta_t, \label{Eq:omega.t.system.product} \\
    \frac{\partial}{\partial t} \Upsilon_t 
    &= \cL_{\nabla_t (\log|\Upsilon_t|_{\omega_t})} \Upsilon_t + \gamma_t. \label{Eq:Upsilon.t.system.product} 
    \end{align}
    where $\beta_t\in\Omega^2(N)$, $\gamma_t\in\Omega^3(N)$ with initial conditions $\omega_0 = \omega$, $\Upsilon_0 = \Upsilon$, and let $\{\varphi_t\}$ be the family $\rG_2$-structures given by
\begin{equation} \label{Eq:varphi.t.product}
    \varphi_t 
    = \Re \left( \frac{1}{|\Upsilon|_{\omega_t}} \Upsilon \right) + |\Upsilon|_{\omega_t} \rd r \wedge \omega_t
\end{equation}
    Then $\{\varphi_t\}$ is a solution of the modified Laplacian coflow \eqref{Eq:Modified.coflow} with constant $A$  if, and only if, 
\begin{align}
    \beta_t \wedge \omega_t 
    &= -A \frac{1}{|\Upsilon|_{\omega}} \rd (\log|\Upsilon_t|_{\omega_t}) \wedge \Re (\Upsilon_t), \label{Eq:beta.t.product} \\
    \Im (\gamma_t) 
    &= A|\Upsilon_t|_{\omega_t}  \rd (\log|\Upsilon_t|_{\omega_t}) \wedge \omega_t. \label{Eq:gamma.t.product}
\end{align}
\end{theorem}

\begin{proof}
    The family of $\rG_2$-structures defined by \eqref{Eq:varphi.t.product} has dual $4$-forms $\psi_t$ given by \eqref{Eq:psi.product}. Since the radial coordinate $r$ on $S^1$ does not depend on $t$, its evolution equation is
\begin{align} 
    \frac{\partial}{\partial t} \psi_t 
    &= \frac{\partial}{\partial t} \left( - \rd r \wedge \Im \Upsilon_t + \frac{1}{2} \omega_t^2 \right)
    = - \rd r \wedge \left( \frac{\partial}{\partial t} \Im \Upsilon_t \right) + \frac{1}{2} \left( \frac{\partial}{\partial t} \omega_t^2 \right) \nonumber \\
    &{= \cL_{\nabla_t (\log|\Upsilon_t|_{\omega_t})} \left( - \rd r \wedge \Im \Upsilon_t - \frac{1}{2} \omega_t^2 \right) - \rd r \wedge \Im(\gamma_t) + \beta_t \wedge \omega_t,} \label{Eq:Laplace.product.1}
\end{align}
    {where we have used \eqref{Eq:omega.t.system.product} and \eqref{Eq:Upsilon.t.system.product}.}
    
    Next, applying Lemma \ref{Lem:Laplacian.psi.product}, \eqref{Eq:d.varphi.product} and the fact {that $(\tau_0)_t = 0$ along the modified Laplacian coflow \eqref{Eq:Modified.coflow} under our Ansatz}, we obtain
    \begin{align} \label{Eq:Laplace.product.2}
    \begin{split}
            \frac{\partial}{\partial t} \left( - \rd r \wedge \Im \Upsilon_t + \frac{1}{2} \omega_t^2 \right) &= \cL_{\nabla_t (\log|\Upsilon_t|_{\omega_t})} \left( - \rd r \wedge \Im \Upsilon_t - \frac{1}{2} \omega_t^2 \right) \\
            &\mathcolor{blue}{-A \frac{1}{|\Upsilon_t|_{\omega_t}} \rd (\log|\Upsilon_t|_{\omega_t}) \wedge \Re (\Upsilon_t) + A|\Upsilon_t|_{\omega_t} \rd (\log|\Upsilon_t|_{\omega_t}) \wedge \rd r \wedge \omega_t},
        \end{split}
    \end{align}
    where the terms in \textcolor{blue}{blue} correspond to the additional term stemming from the {DeTurck} modification. 
    
    {Comparing \eqref{Eq:Laplace.product.1} and \eqref{Eq:Laplace.product.2},} we get
    \begin{align*}
        &- \rd r \wedge \Im (\gamma_t) + \beta_t \wedge \omega_t = \mathcolor{blue}{-A\frac{1}{|\Upsilon_t|_{\omega_t}} \rd (\log|\Upsilon_t|_{\omega_t}) \wedge \Re (\Upsilon_t) + A|\Upsilon_t|_{\omega_t} \rd (\log|\Upsilon_t|_{\omega_t}) \wedge dr \wedge \omega_t}.
    \end{align*}
    Since the radial coordinate is independent of $t$, we can contract by $\partial_r$, which yields  \eqref{Eq:gamma.t.product}. Using \eqref{Eq:gamma.t.product} and the above equation, we get \eqref{Eq:beta.t.product}. 
\end{proof}

As a corollary, we obtain restrictions on the forms $\beta_t$ and $\gamma_t$, assuming that the complex structure $J$ is to stay fixed along the flow.

\begin{cor} \label{cor:Modified.coflow.product.fixed.J}
    Let $\{\varphi_t\}$ of the form \eqref{Eq:varphi.t.product} be a solution to the modified Laplacian coflow,  such that the associated $\SU(3)$-structures satisfy \eqref{Eq:omega.t.system.product} and \eqref{Eq:Upsilon.t.system.product}. If the complex structure $J$ on $N$ remains fixed along the flow, then
\begin{align}
    \beta_t^{(1,1)} &= 0, \\
    \Im (\gamma_t)^{(0,3) \oplus (3,0)} &= 0.
\end{align}
\end{cor}
\begin{proof}
Since $\{\varphi_t\}$ is a solution, we must have that $\omega_t$ and $\Upsilon_t$ satisfy \eqref{Eq:omega.t.system.product} and \eqref{Eq:Upsilon.t.system.product}. If the complex structure $J$ is fixed, we must have $\frac{\partial}{\partial t} \omega_t\in \Omega^{1,1}$. We see that the RHS of \eqref{Eq:omega.t.system.product} has bidegree $(1,3) \oplus (3,1)$. Since $\omega_t \in \Omega^{1,1}$ it follows that the $(1,1)$-part of $\beta_t$ must vanish. A similar analysis shows that the $(3,0) \oplus (0,3)$-part of $\Im (\gamma_t)$ must also vanish.
\end{proof}

\section{Flows on contact Calabi--Yau 7-Manifolds} \label{Sect:cCY}

We now extend the ideas of \cite{Picard2022flows} to contact Calabi--Yau (cCY) manifolds, and investigate both the Laplacian coflow and the modified coflow on those spaces. We employ the approach of Tomassini--Vezzoni \cite{Tomassini2008} and Habib--Vezzoni \cite{Habib2015} for the geometry of Sasakian manifolds satisfying $\Hol(\nabla)\subseteq \SU(n)$ in $\rG_2$-geometry; see also \cite{Calvo-Andrade2020}.

\begin{definition} \label{Def:cCY}
 	A \emph{contact Calabi--Yau} ($cCY^7$) $7$-manifold is a quadruple   $(M^7,\eta,\Phi,\Upsilon)$ such that
    \begin{itemize}
	   \item $(M,\xi,\eta,\Phi,g)$ is a $7$-dimensional Sasakian manifold with Reeb vector field $\xi$ and contact form $\eta$ and vanishing first basic Chern class $c_1^B(M) = 0$, see Appendix \ref{Ap.Sasakian};
	
	   \item $\Upsilon$ is a nowhere vanishing transverse form on $\cD = \ker         \eta$ of type $(3,0)$, with
	   $$\frac{\omega^3}{3!} = \vol_{\cD} = \frac{i}{8} \frac{1}{|\Upsilon|_\omega^2} \Upsilon \wedge \oep,\quad \rd \Upsilon=0, $$
	   where $\omega = \rd \eta$.
    We also define
	   $$ \Re \Upsilon := \frac{\Upsilon+\oep}{2}, \quad \Im \Upsilon :=             \frac{\Upsilon-\oep}{2i}. $$  	
	\end{itemize}

 We refer to $(\omega, \Upsilon)$ as a transverse $\SU(3)$-structure and the norm $|\Upsilon|_\omega$ is constant when $\omega$ is transverse Ricci-flat.
\end{definition}.

\begin{remark} \label{Rem:Sasakian}
	A contact Calabi--Yau manifold $(M,g,\eta,\Upsilon)$ has transverse Calabi--Yau geometry on the distribution $\cD = \ker \eta$, in the sense of foliations, given by $g|_{\cD}$, $\omega$ and $\Upsilon$. When the Sasakian structure is regular or quasi-regular, $M$ is an $S^1$-(orbi)bundle over a Calabi--Yau orbifold $\cZ = M /\cF_\xi$ where $\cF_\xi$ is the foliation obtained from the Reeb vector field $\xi$.
    The Sasakian geometry can also be irregular, and in this case there is no $S^1$-fibration structure on $M$ compatible with the contact Calabi--Yau geometry.
\end{remark}

\subsection{Preliminaries on \texorpdfstring{$cCY^7$}{cCY7}}

We recall how to relate the cCY geometry in $7$ dimensions to $\rG_2$-geometry, cf.~\cite{Habib2015}*{Corollary 6.8} and \cite{lotay2021}. 
\begin{prop} \label{Prop:G2.cCY}
	Let $(M^7,\eta,\Phi,\Upsilon)$ be a contact Calabi--Yau 7-manifold {with Reeb vector field $\xi$}. Then $M$ carries a 1-parameter of coclosed $\rG_2$-structures defined by
    \begin{align} \label{Eq:psi.ccY.1.parameter}
	   \varphi = \Re \Upsilon + \epsilon \eta \wedge \omega,
    \end{align}  
	for $\epsilon > 0$, where $\omega = \rd \eta$ is transverse Ricci-flat. Furthermore, $\rd \varphi = \epsilon \omega^2$ and $\varphi$ is coclosed, ie.  $\rd \psi = 0$. 
\end{prop}

The metric $g$ and the transverse symplectic form $\omega = \rd \eta$ on $(M,\eta,\Phi,\Upsilon)$ can be written locally as 
\begin{align}
    g = \eta^2 + g_{p\overline{q}} \rd z^p \rd \overline{z}^{q}, \qquad \rd \eta &= 2i g_{p\overline{q}} \rd z^p \wedge \rd \overline{z}^q,\qquad \Upsilon = u \rd z^1 \wedge \rd z^2 \wedge \rd z^3,    
\end{align}
where the $g_{p\overline{q}}$ and $u$ are all basic functions{, that is $\cL_\xi g_{p\overline{q}} = \cL_\xi u = 0$}. Moreover, we obtain a basic function defined by  
\begin{align}
    |\Upsilon|^2_{\omega} = \frac{|u|^2}{\det (g)_{p\overline{q}}}.
\end{align}

We obtain a coclosed $\rG_2$-structure given by 
\begin{equation} \label{Eq:varphi.cCY}
    \varphi = \Re \left( \frac{1}{|\Upsilon|_{\omega}} \Upsilon \right) + |\Upsilon|_{\omega} \eta \wedge \omega.
\end{equation}
In this case, the associated metric on $M$ is 
\begin{equation} \label{Eq:g.cCY}
    g = |\Upsilon|_{\omega}^2 \eta^2 + g|_{\cD},
\end{equation}
the volume form is
\begin{equation} \label{Eq:vol.cCY}
    \vol = |\Upsilon|_{\omega} \eta \wedge \vol|_{\cD},
    \qwithq 
    \vol|_{\cD} = \frac{\omega^3}{3!},
\end{equation}
and the dual $4$-form $\psi$ is 
\begin{equation} 
\label{Eq:psi.cCY}
    \psi = - \eta \wedge \Im \Upsilon + \frac{1}{2} \omega^2.
\end{equation}
We recall that $\omega$ and $\Upsilon$ are closed, and the contact form $\eta$ satisfies $d\eta = \omega$. It follows that $\psi$ is closed, ie. $\varphi$ is coclosed.
The Reeb vector field $\xi$ generates a $1$-dimensional foliation $\cF_{\xi}$, whose orientation induces a basic Hodge operator 
\begin{equation}
    \ast_B : \Lambda^k_B(M) \rightarrow \Lambda^{6-k}_B(M)
\end{equation}
in the usual way. Then, for $\alpha \in \Omega_B^k(M)$,
\begin{equation}
\label{Eq:Basic.Hodge.star}
     (\ast_B)^2 \alpha 
     = (-1)^k \alpha, \quad \ast_B \Re (\Upsilon) = \Im \Upsilon, \quad \ast_B \omega = \frac{1}{2} \omega^2,
\end{equation} 
This relates to the standard Hodge operator of the $7$-dimensional metric \eqref{Eq:g.cCY}, acting on $\alpha\in \Omega_B^k(M)\hookrightarrow \Omega^k(M)$, by
\begin{align}
    \ast \alpha 
    &= (-1)^k |\Upsilon|_{\omega} \eta \wedge \ast_B \alpha, \label{Eq:Basic.Hodge.star.2} \\
    \ast (\eta \wedge \alpha) 
    &= \frac{1}{|\Upsilon|_{\omega}} \ast_B \alpha \label{Eq:Basic.Hodge.star.3}
\end{align}

We compute the torsion forms of the $\rG_2$-structure \eqref{Eq:varphi.cCY}, distinguishing in \textcolor{red}{red}  terms that arise from the non-trivial topology of the $cCY^7$, compared to the product $CY^3\times S^1$.

\begin{prop} 
\label{Prop:Torsion.ccY}
    Let $(M^7, \eta, \Phi, \Upsilon)$ be a contact Calabi--Yau $7$-manifold, with $\rG_2$-structure $\varphi$   defined by \eqref{Eq:varphi.cCY}. Then the torsion forms of $\varphi$ are given by
\begin{equation}
    \tau_0 = \mathcolor{red}{\frac{6}{7} |\Upsilon|_{\omega}}, 
    \quad \tau_1 = 0, 
    \quad \tau_2 = 0, 
\end{equation}
    and 
\begin{equation}
    \tau_3 = (\nabla \log|\Upsilon|_{\omega}) \lrcorner \left (-\eta \wedge \Im \Upsilon +\frac{1}{2} \omega^2 \right)  \mathcolor{red}{-\frac{6}{7} \Re \Upsilon + \frac{8}{7} |\Upsilon|^2_{\omega} \eta \wedge \omega.} \label{Eq:tau_3.ccY}
\end{equation}
\end{prop}

\begin{proof}
    Since $\varphi$ is coclosed, we have $\tau_1=0$ and $\tau_2=0$. We now compute $\tau_0$, as follows.
    \begin{align}
        \rd \varphi 
        &= \rd \left( \Re \left( \frac{1}{|\Upsilon|_{\omega}} \Upsilon \right) + |\Upsilon|_{\omega} \eta \wedge \omega \right) \nonumber \\
        &= -\frac{1}{|\Upsilon|_{\omega}} \rd (\log|\Upsilon|_{\omega}) \wedge \Re (\Upsilon) + |\Upsilon|_{\omega} \rd (\log|\Upsilon|_{\omega}) \wedge \eta \wedge \omega \mathcolor{red}{+ |\Upsilon|_{\omega} \omega^2}. \label{Eq:d.varphi.cCY}
    \end{align}
    Taking the Hodge star of both sides we obtain
    \begin{align}
        \ast \rd \varphi 
        &=-\frac{1}{|\Upsilon|_{\omega}} \ast \left( \log|\Upsilon|_{\omega} \wedge \Re (\Upsilon) \right) + |\Upsilon|_{\omega} \ast \left( \rd (\log|\Upsilon|_{\omega}) \wedge \eta \wedge \omega \right) \mathcolor{red}{+ |\Upsilon|_{\omega} \ast \omega^2} \nonumber \\
        &= \frac{1}{|\Upsilon|_{\omega}} (\rd \log|\Upsilon|_{\omega})^{\sharp} \lrcorner \ast (\Re (\Upsilon)) - |\Upsilon|_{\omega} (\rd \log|\Upsilon|_{\omega})^{\sharp} \lrcorner \ast (\eta \wedge \omega) \mathcolor{red}{+ |\Upsilon|_{\omega} \ast \omega^2} \nonumber \\
        &= (\nabla \log|\Upsilon|_{\omega}) \lrcorner \left( -\eta \wedge \Im \Upsilon - \frac{1}{2} \omega^2 \right) \mathcolor{red}{+ 2|\Upsilon|_{\omega}^2 \eta \wedge \omega.} \label{Eq:star.d.varphi.cCY}
    \end{align}
    Using  \eqref{Eq:d.varphi.cCY}, we find 
    \begin{align*}
        \tau_0 &= \frac{1}{7} \ast (\varphi \wedge d\varphi) = \mathcolor{red}{\frac{6}{7} \ast \left( |\Upsilon|^2_{\omega} \eta \wedge \frac{\omega^3}{3!} \right) = \frac{6}{7} |\Upsilon|_{\omega}.}
    \end{align*}
    Finally, we compute $\tau_3$, from \eqref{Eq:star.d.varphi.cCY}:
    \begin{align}
        \tau_3 
        &=\ast \rd \varphi - \tau_0 \varphi \nonumber \\
        &= (\nabla \log|\Upsilon|_{\omega}) \lrcorner \left (-\eta \wedge \Im \Upsilon - \frac{1}{2} \omega^2 \right) \mathcolor{red}{+ 2|\Upsilon|_{\omega}^2 \eta \wedge \omega - \frac{6}{7} |\Upsilon|_{\omega} \left( \Re \frac{1}{|\Upsilon|_{\omega}} \Upsilon \right) - \frac{6}{7} |\Upsilon|^2_{\omega} \eta \wedge \omega} \nonumber \\
        &= (\nabla \log|\Upsilon|_{\omega}) \lrcorner \left( -\eta \wedge \Im \Upsilon - \frac{1}{2} \omega^2 \right) \mathcolor{red}{- \frac{6}{7} \Re \Upsilon + \frac{8}{7} |\Upsilon|^2_{\omega} \eta \wedge \omega.} \nonumber \qedhere   
    \end{align}
\end{proof}

\begin{prop} \label{Prop:Laplacian.psi.cCY}
    Let $(M^7, \eta, \Phi, \Upsilon)$ be a contact Calabi--Yau $7$-manifold, with $\rG_2$-structure $\varphi$   defined by \eqref{Eq:varphi.cCY}.
    Then the Hodge Laplacian of $\psi = \ast \varphi$ is
    \begin{equation} \label{Eq:Laplace.psi.cCY}
        \Delta \psi = \cL_{\nabla (\log|\Upsilon|_{\omega})} \left( -\eta \wedge \Im \Upsilon - \frac{1}{2} \omega^2 \right) \mathcolor{red}{+ 4|\Upsilon|_{\omega}^2 \rd \left( \log|\Upsilon|_{\omega} \right) \eta \wedge \omega + 2|\Upsilon|_{\omega}^2 \omega^2.}
    \end{equation} 
\end{prop}

\begin{proof}
    Since $\varphi$ is coclosed,  the Hodge Laplacian is given by $\Delta \psi = \rd \rd^{\ast} \psi = \rd \ast \rd \varphi$. We recall Cartan's formula $\cL_Y \alpha = \rd (Y \lrcorner \alpha) + Y \lrcorner (\rd \alpha)$, for $\alpha \in \Omega^k(M)$ and $Y \in \fX(M)$. Using the fact that $\omega$ and $\Upsilon$ are closed, together with \eqref{Eq:d.varphi.cCY} and \eqref{Eq:star.d.varphi.cCY}, we get 
    \begin{align*}
        \Delta \psi 
        &= \rd \ast \rd \varphi = \cL_{\nabla (\log|\Upsilon|_{\omega})} \left( -\eta \wedge \Im \Upsilon - \frac{1}{2} \omega^2 \right) \mathcolor{red}{+ 2 \rd \left( |\Upsilon|_{\omega}^2 \eta \wedge \omega \right)}  \\
        &= \cL_{\nabla (\log|\Upsilon|_{\omega})} \left( -\eta \wedge \Im \Upsilon - \frac{1}{2} \omega^2 \right) \mathcolor{red}{+ 4|\Upsilon|_{\omega}^2 \rd \left( \log|\Upsilon|_{\omega} \right) \wedge \eta \wedge \omega + 2|\Upsilon|^2_{\omega} \omega^2.}
        \qedhere
    \end{align*}
\end{proof}

\subsection{The Laplacian coflow} \label{Subsect:Laplacian.coflow.cCY}

Let $(M,\eta,\Phi,\Upsilon)$ be a contact Calabi--Yau $7$-manifold with $\rG_2$-structure $\varphi$ defined by  \eqref{Eq:varphi.cCY}. We now consider the Laplacian coflow in this setting.  Define a family of contact forms by $\eta_t = \eta + \rd^c f_t$, where each $f_t$ is a basic function. This in turn defines a family of transverse K\"{a}hler structures 
$$\omega_t = \rd \eta_t = \omega + \rd \rd^c f_t.$$ 
We note that the endomorphism $\Phi_t$ varies, but the Reeb vector field $\xi$, the space of basic forms $\Omega^\bullet_B(M)$, and the transverse complex structure $J$ remain constant under these deformations (see Appendix \ref{Ap.Sasakian}).

We have the following result, analogous to Corollary \ref{cor:Modified.coflow.product.fixed.J}, describing the effects of fixing the transverse complex structure $J$.

\begin{theorem} \label{Thm:Laplacian.coflow.cCY.fixed.J}
    Let $(M,\eta,\Phi,\Upsilon)$ be a contact Calabi--Yau $7$-manifold with transverse K\"{a}hler form $\omega = \rd \eta$ and transverse holomorphic $(3,0)$-form $\Upsilon$. Suppose we have a family of compatible transverse $\SU(3)$-structures $(\omega_t,\Upsilon_t)$ on $M$ satisfying the coupled differential equations 
    \begin{align}
        \frac{\partial}{\partial t} \omega_t &= - \cL_{\nabla_t (\log|\Upsilon_t|_{\omega_t})} \omega_t + \beta_t, \label{Eq:omega.t.system.cCY} \\
        \frac{\partial}{\partial t} \Upsilon_t &= \cL_{\nabla_t (\log|\Upsilon_t|_{\omega_t})} \Upsilon_t + \gamma_t. \label{Eq:Upsilon.t.system.cCY}
    \end{align}
    where $\beta_t \in \Omega^2_B(M)$, $\gamma_t \in \Omega^3_B(M)$ with initial conditions $\omega_0 = \omega$, $\Upsilon_0 = \Upsilon$. Suppose further that there exists a family of basic functions $\{f_t\}$ such that $\omega_t = \omega + \rd \rd^c f_t$, and let $\eta_t := \eta + \rd^c f_t$. 

    Then, the family of $\rG_2$-structures given by
    \begin{equation} \label{Eq:varphi.t.cCY}
        \varphi_t = \Re \left( \frac{1}{|\Upsilon_t|_{\omega_t}} \Upsilon_t \right) + |\Upsilon|_{\omega_t} \eta_t \wedge \omega_t.
    \end{equation}
    is a solution of the Laplacian coflow \eqref{eq: Laplacian.coflow} if, and only if, 
    \begin{align}
    \beta_t \wedge \omega_t &= \mathcolor{red}{2|\Upsilon_t|^2_{\omega_t} \omega_t^2} - \left( \cL_{\nabla (\log|\Upsilon_t|_{\omega_t})} \rd^c f_t \right) \wedge \Im \Upsilon_t \nonumber \\
    &\qquad - \Big( \nabla (\log|\Upsilon_t|_{\omega_t}) \lrcorner \omega \Big) \wedge \Im \Upsilon_t + \left( \frac{\del}{\del t} \rd^c f_t \right) \wedge \Im \Upsilon_t, \label{Eq:beta.t.cCY} \\
    \Im (\gamma_t) &= \mathcolor{red}{4|\Upsilon|_{\omega_t}^2 \rd \left( \log|\Upsilon_t|_{\omega_t} \right) \wedge \omega_t.} \label{Eq:gamma.t.cCY}
 \end{align}
\end{theorem}

\begin{proof}
    The family of $\rG_2$-structures defined by \eqref{Eq:varphi.t.cCY} has associated $4$-form $\psi_t = \ast_t \varphi_t$ given by \eqref{Eq:psi.cCY}, whose evolution equation is  
    \begin{align} 
        \frac{\partial}{\partial t} \psi_t &= \frac{\partial}{\partial t} \left( -\eta_t \wedge \Im \Upsilon_t + \frac{1}{2} \omega_t^2 \right) = -\eta_t \wedge \left( \frac{\partial}{\partial t} \Im \Upsilon \right) + \frac{1}{2} \left( \frac{\partial}{\partial t} \omega_t^2 \right) - \left( \frac{\partial}{\partial t} \eta_t \right) \wedge \Im \Upsilon_t \nonumber \\
        &{= - \eta_t \wedge \left( \cL_{\nabla_t (\log|\Upsilon_t|_{\omega_t})} \Im \Upsilon_t \right) + \cL_{\nabla_t (\log|\Upsilon_t|_{\omega_t})} \left(- \frac{1}{2} \omega_t^2 \right)} \nonumber \\
        &\qquad \qquad {- \eta_t \wedge \Im (\gamma_t) + \beta_t \wedge \omega_t - \left( \frac{\partial}{\partial t} \eta_t \right) \wedge \Im \Upsilon_t,} \label{Eq:Laplace.cCY.1}
    \end{align}
    {where we have used \eqref{Eq:omega.t.system.cCY} and \eqref{Eq:Upsilon.t.system.cCY}.}
    
    Thus, applying Proposition \ref{Prop:Laplacian.psi.cCY} to the Laplacian coflow \eqref{Eq:Laplace.psi.cCY}, we obtain
    \begin{align} \label{Eq:Laplace.cCY.2}
        &\frac{\partial}{\partial t} \left( -\eta_t \wedge \Im \Upsilon_t + \frac{1}{2} \omega_t^2 \right) \nonumber \\
        &= \cL_{\nabla_t (\log|\Upsilon_t|_{\omega_t})} \left( -\eta_t \wedge \Im \Upsilon_t - \frac{1}{2} \omega_t^2 \right) \mathcolor{red}{+ 4|\Upsilon_t|_{\omega_t}^2 \rd \left( \log|\Upsilon_t|_{\omega_t} \right)\wedge \eta_t \wedge \omega_t + 2|\Upsilon|^2_{\omega_t} \omega_t^2} .
    \end{align}
    
    Applying Cartan's magic formula to the Lie derivative term, we obtain
    \begin{align}
        &\cL_{\nabla_t (\log|\Upsilon_t|_{\omega_t})} (\eta_t \wedge \Im \Upsilon) \nonumber \\
        &= \left( \cL_{\nabla_t (\log|\Upsilon_t|_{\omega_t})} \eta_t \right) \wedge \Im \Upsilon_t + \eta_t \wedge \left( \cL_{\nabla (\log|\Upsilon_t|_{\omega_t})} \Im \Upsilon \right) \nonumber \\
        &= \rd \Big( \nabla_t (\log|\Upsilon_t|_{\omega_t}) \lrcorner \eta_t \Big) \wedge \Im \Upsilon_t + \Big( \nabla_t (\log|\Upsilon_t|_{\omega_t}) \lrcorner \omega_t \Big) \wedge \Im \Upsilon_t + \eta_t \wedge \left( \cL_{\nabla_t (\log|\Upsilon_t|_{\omega_t})} \Im \Upsilon_t \right).
    \end{align}
    Since $\rd \left(\log|\Upsilon_t|_{\omega_t} \right)$ is a basic function (recall that the Reeb vector field $\xi$ is fixed along these deformations) and $\eta_t = \eta + \rd^c f_t$, the above expression becomes
    \begin{align} \label{Eq:Lie.derivative.cCY}
        &\cL_{\nabla_t (\log|\Upsilon_t|_{\omega_t})} (\eta_t \wedge \Im \Upsilon_t) \nonumber \\
        &= \rd \Big( \nabla_t (\log|\Upsilon_t|_{\omega_t}) \lrcorner \rd^c f_t \Big) \wedge \Im \Upsilon_t + \Big( \nabla_t (\log|\Upsilon_t|_{\omega_t}) \lrcorner \omega_t \Big) \wedge \Im \Upsilon_t + \eta_t \wedge \left( \cL_{\nabla_t (\log|\Upsilon_t|_{\omega_t})} \Im \Upsilon_t \right) \nonumber \\
        &= \left (\cL_{\nabla_t (\log|\Upsilon_t|_{\omega_t})} \rd^c f_t \right) \wedge \Im \Upsilon_t + \Big( \nabla_t (\log|\Upsilon_t|_{\omega_t}) \lrcorner \omega \Big) \wedge \Im \Upsilon_t + \eta_t \wedge \left( \cL_{\nabla_t (\log|\Upsilon_t|_{\omega_t})} \Im \Upsilon_t \right).
    \end{align}
    
    {Comparing \eqref{Eq:Laplace.cCY.1} and \eqref{Eq:Laplace.cCY.2},} and using \eqref{Eq:Lie.derivative.cCY}, we get
    \begin{align} \label{Eq:beta.t.gamma.t.cCY}
        &- \left( \frac{\partial}{\partial t} \eta_t \right) \wedge \Im \Upsilon_t - \eta_t \wedge \Im (\gamma_t) + \beta_t \wedge \omega_t \nonumber \\
        &= \mathcolor{red}{4|\Upsilon_t|_{\omega_t}^2  \rd \left( \log|\Upsilon_t|_{\omega_t} \right) \wedge \eta_t \wedge \omega_t + 2|\Upsilon_t|^2_{\omega_t} \omega_t^2} \nonumber \\
        &\qquad - \left( \cL_{\nabla_t (\log|\Upsilon_t|_{\omega_t})} \rd^c f_t \right) \wedge \Im \Upsilon_t - \Big( \nabla_t (\log|\Upsilon_t|_{\omega_t}) \lrcorner \omega \Big) \wedge \Im \Upsilon_t.
    \end{align}

    Since the Reeb vector field $\xi$ is constant with respect to $t$, we can contract by $\xi$ and obtain
    \begin{align}
        -\left[\xi \lrcorner \left( \frac{\partial}{\partial t} \eta_t \right) \right] \wedge \Im \Upsilon_t - (\xi \lrcorner \, \eta_t) \wedge \Im (\gamma_t) = \mathcolor{red}{- 4 |\Upsilon_t|_{\omega_t}^2 (\xi \lrcorner \, \eta_t) \wedge \rd \left( \log|\Upsilon_t|_{\omega_t} \right) \wedge \omega_t.}
    \end{align}
     Using $\eta_t = \eta + \rd^c f_t$, this simplifies to
    \begin{align}\label{Eq.Im.gamma}
        \Im (\gamma_t) &= \mathcolor{red}{4|\Upsilon_t|_{\omega_t}^2 \rd \left( \log|\Upsilon_t|_{\omega_t} \right) \wedge \omega_t.} 
    \end{align}
    {Substituting \eqref{Eq.Im.gamma} into \eqref{Eq:beta.t.gamma.t.cCY} yields:} 
    \begin{equation}
        \beta_t \wedge \omega_t = \mathcolor{red}{2|\Upsilon_t|^2_{\omega_t} \omega_t^2} - \left( \cL_{\nabla_t (\log|\Upsilon_t|_{\omega_t})} \rd^c f_t \right) \wedge \Im \Upsilon_t - \Big( \nabla_t (\log|\Upsilon_t|_{\omega_t}) \lrcorner \omega \Big) \wedge \Im \Upsilon_t + \left( \frac{\del}{\del t} \rd^c f_t \right) \wedge \Im \Upsilon_t,
    \end{equation}
    which concludes the proof.
\end{proof}

\subsection{The modified Laplacian coflow} 
\label{Subsect:Modified.cof.ccY}

We now turn our attention to the modified Laplacian coflow \eqref{Eq:Modified.coflow}. Recall that if $(\omega, \Upsilon)$ is a transverse $\SU(3)$-structure on a contact Calabi--Yau $7$-manifold $(M^7, \eta, \Phi, \Upsilon)$, we can define a coclosed $\rG_2$-structure by
\begin{align}
    \varphi = \Re \left( \frac{1}{|\Upsilon|_{\omega}} \Upsilon \right) + |\Upsilon|_{\omega} \eta \wedge \omega. \nonumber
\end{align}
Such a $\rG_2$-structure has $\tau_0 = \frac{6}{7} |\Upsilon|_{\omega}$, hence the added term in the modified coflow with constant $A$ would be
\begin{align}
    \rd \left( \left( A - \frac{7}{2} \tau_0 \right) \varphi \right) &= - \frac{A}{|\Upsilon|_{\omega}} \rd (\log |\Upsilon|_{\omega}) \wedge \Re (\Upsilon) + A |\Upsilon|_{\omega} \rd (\log |\Upsilon|_{\omega}) \wedge \eta \wedge \omega + A |\Upsilon|_{\omega} \omega^2 \nonumber \\
    &\qquad - 6 |\Upsilon|^2_{\omega} \rd (\log |\Upsilon|_{\omega}) \wedge \eta \wedge \omega - 3 |\Upsilon|^2_{\omega} \omega^2. \label{Eq:Modified.coflow.extra.cCY}
\end{align}

Let $\eta_t$ and $\omega_t$ evolve as in the previous subsection, via \eqref{Eq:omega.t.system.cCY} and \eqref{Eq:Upsilon.t.system.cCY}. As before, the Reeb vector field $\xi$ and transverse complex structure stay fixed. We now obtain a similar result to Theorem \ref{Thm:Laplacian.coflow.cCY.fixed.J} for the modified coflow. As before, the \textcolor{red}{red} terms are from the non-trivial topology. We shall also denote terms derived from the de Turck modification in \textcolor{blue}{blue}, and terms coming from a combination of both the topology and the modification in \textcolor{purple}{purple}.

\begin{theorem} \label{thm:modified.coflow.cCY.fixed.J}
    Let $(M^7,\eta,\Phi,\Upsilon)$ be a contact Calabi--Yau $7$-manifold, with transverse K\"{a}hler form $\omega = \rd \eta$ and transverse holomorphic $(3,0)$-form $\Upsilon$. Suppose we have a family of compatible transverse $\SU(3)$-structures $(\omega_t,\Upsilon_t)$ on $M$ satisfying the coupled differential equations: 
    \begin{align}
        \frac{\partial}{\partial t} \omega_t &= - \cL_{\nabla_t (\log|\Upsilon_t|_{\omega_t})} \omega_t + \beta_t \label{Eq:omega.t.system.cCY-2} \\
        \frac{\partial}{\partial t} \Upsilon_t &= \cL_{\nabla_t (\log|\Upsilon_t|_{\omega_t})} \Upsilon_t + \gamma_t, \label{Eq:Upsilon.t.system.cCY-2}
    \end{align}
    where $\beta_t \in \Omega^2_B(M)$, $\gamma_t \in \Omega^3_B(M)$, with initial conditions $\omega_0 = \omega$, $\Upsilon_0 = \Upsilon$. Suppose further that there exists a family of basic functions $f_t$ such that $\omega_t = \omega + \rd \rd^c f_t$, and let $\eta_t := \eta + \rd^c f_t$. 

    Then the family of $\rG_2$-structures given by
    \begin{equation}
    \label{Eq:varphi.t.cCY-2}
        \varphi_t = \Re \left( \frac{1}{|\Upsilon_t|_{\omega_t}} \Upsilon_t \right) + |\Upsilon|_{\omega_t} \eta_t \wedge \omega_t
    \end{equation}
    is a solution of the modified Laplacian coflow \eqref{Eq:Modified.coflow} if, and only if,
    \begin{align}
        \beta_t \wedge \omega_t 
        &= \mathcolor{purple}{-|\Upsilon_t|^2_{\omega_t} \omega_t^2} \mathcolor{blue}{+ A |\Upsilon_t|_{\omega_t} \omega_t^2 - \frac{A}{|\Upsilon_t|_{\omega_t}} \rd (\log |\Upsilon_t|_{\omega_t}) \wedge \Re \Upsilon_t} \nonumber \\
        &\qquad  - \left( \cL_{\nabla_t (\log|\Upsilon_t|_{\omega_t})} \rd^c f_t \right) \wedge \Im \Upsilon_t - \Big( \nabla_t (\log|\Upsilon_t|_{\omega_t}) \lrcorner \omega \Big) \wedge \Im \Upsilon_t + \left( \frac{\del}{\del t} \rd^c f_t \right) \wedge \Im \Upsilon_t, \label{Eq:beta.t.cCY-2} \\
        \Im (\gamma_t) 
        &= \mathcolor{purple}{-2|\Upsilon_t|_{\omega_t}^2 \rd \left( \log|\Upsilon_t|_{\omega_t} \right) \wedge \omega_t} \mathcolor{blue}{+ A |\Upsilon_t|_{\omega_t} \rd \left( \log|\Upsilon_t|_{\omega_t} \right) \wedge \omega_t.} \label{Eq:gamma.t.cCYY-2}
    \end{align}
\end{theorem}
\begin{proof}[Sketch of Proof]
    The proof is similar to that of Theorem \ref{Thm:Laplacian.coflow.cCY.fixed.J}, one just has to incorporate the extra terms computed in \eqref{Eq:Modified.coflow.extra.cCY}.  Recall that the dual $4$-form $\psi_t$ given by the expression
    \begin{align}
        \psi_t = -\eta_t \wedge \Im \Upsilon_t + \frac{1}{2} \omega_t^2,
    \end{align}
    and so
    \begin{align}
        \frac{\partial}{\partial t} \psi_t = -\eta_t \wedge \left( \frac{\partial}{\partial t} \Im \Upsilon \right) + \frac{1}{2} \left( \frac{\partial}{\partial t} \omega_t^2 \right) - \left( \frac{\partial}{\partial t} \eta_t \right) \wedge \Im \Upsilon_t.
    \end{align}
    
    Applying Proposition \ref{Prop:Laplacian.psi.cCY} and \eqref{Eq:Modified.coflow.extra.cCY}, the modified coflow implies the evolution equation
    \begin{align}
        &\frac{\partial}{\partial t} \left( -\eta_t \wedge \Im \Upsilon_t + \frac{1}{2} \omega_t^2 \right) \nonumber \\
        &= \cL_{\nabla_t (\log|\Upsilon_t|_{\omega_t})} \left( -\eta_t \wedge \Im \Upsilon_t - \frac{1}{2} \omega_t^2 \right) \mathcolor{red}{+ 4|\Upsilon_t|_{\omega_t}^2 \rd \left( \log|\Upsilon_t|_{\omega_t} \right) \eta_t \wedge \omega_t + 2|\Upsilon|^2_{\omega_t} \omega_t^2} \nonumber \\
        &\qquad \mathcolor{blue}{- \frac{A}{|\Upsilon_t|_{\omega_t}} \rd (\log |\Upsilon_t|_{\omega_t}) \wedge \Re (\Upsilon_t) + A |\Upsilon_t|_{\omega_t} \rd (\log |\Upsilon_t|_{\omega_t}) \wedge \eta_t \wedge \omega_t + A |\Upsilon_t|_{\omega_t} \omega_t^2} \nonumber \\
        &\qquad \mathcolor{blue}{- 6 |\Upsilon_t|^2_{\omega_t} \rd (\log |\Upsilon_t|_{\omega_t}) \wedge \eta_t \wedge \omega_t - 3 |\Upsilon_t|^2_{\omega_t} \omega_t^2,} \nonumber \\
        &= \cL_{\nabla_t (\log|\Upsilon_t|_{\omega_t})} \left( -\eta_t \wedge \Im \Upsilon_t - \frac{1}{2} \omega_t^2 \right) \mathcolor{purple}{- 2|\Upsilon_t|_{\omega_t}^2 \rd \left( \log|\Upsilon_t|_{\omega_t} \right) \eta_t \wedge \omega_t - |\Upsilon|^2_{\omega_t} \omega_t^2} \nonumber \\
        &\qquad \mathcolor{blue}{- \frac{A}{|\Upsilon_t|_{\omega_t}} \rd (\log |\Upsilon_t|_{\omega_t}) \wedge \Re (\Upsilon_t) + A |\Upsilon_t|_{\omega_t} \rd (\log |\Upsilon_t|_{\omega_t}) \wedge \eta_t \wedge \omega_t + A |\Upsilon_t|_{\omega_t} \omega_t^2.} \nonumber
    \end{align}
    
    By the proof of Theorem \ref{Thm:Laplacian.coflow.cCY.fixed.J}, we have
    \begin{align}
        &\cL_{\nabla_t (\log|\Upsilon_t|_{\omega_t})} (\eta_t \wedge \Im \Upsilon_t) \nonumber \\
        &= \left (\cL_{\nabla_t (\log|\Upsilon_t|_{\omega_t})} \rd^c f_t \right) \wedge \Im \Upsilon_t + \Big( \nabla_t (\log|\Upsilon_t|_{\omega_t}) \lrcorner \omega \Big) \wedge \Im \Upsilon_t + \eta_t \wedge \left( \cL_{\nabla_t (\log|\Upsilon_t|_{\omega_t})} \Im \Upsilon_t \right).
    \end{align}

    Applying the Ans\"{a}tze \eqref{Eq:omega.t.system.cCY-2} and \eqref{Eq:Upsilon.t.system.cCY-2}, we are left with
    \begin{align} 
        &- \left( \frac{\partial}{\partial t} \eta_t \right) \wedge \Im \Upsilon_t - \eta_t \wedge \Im (\gamma_t) + \beta_t \wedge \omega_t \nonumber \\
        &= \mathcolor{purple}{- 2|\Upsilon_t|_{\omega_t}^2 \rd \left( \log|\Upsilon_t|_{\omega_t} \right) \wedge \eta_t \wedge \omega_t - |\Upsilon_t|^2_{\omega_t} \omega_t^2} \nonumber \\
        &\qquad \mathcolor{blue}{- \frac{A}{|\Upsilon_t|_{\omega_t}} d(\log |\Upsilon_t|_{\omega_t}) \wedge \Re (\Upsilon_t) + A |\Upsilon_t|_{\omega_t} \rd (\log |\Upsilon_t|_{\omega_t}) \wedge \eta_t \wedge \omega_t + A |\Upsilon_t|_{\omega_t} \omega_t^2} \nonumber \\
        &\qquad - \left( \cL_{\nabla_t (\log|\Upsilon_t|_{\omega_t})} \rd^c f_t \right) \wedge \Im \Upsilon_t - \Big( \nabla_t (\log|\Upsilon_t|_{\omega_t}) \lrcorner \omega \Big) \wedge \Im \Upsilon_t.  \label{Eq:beta.t.gamma.t.cCY-modified}
    \end{align}

    Contracting with the Reeb vector field $\xi$, we get
    \begin{align}
        \Im (\gamma_t) &= \mathcolor{purple}{-2|\Upsilon_t|_{\omega_t}^2 \rd \left( \log|\Upsilon_t|_{\omega_t} \right) \wedge \omega_t} \mathcolor{blue}{+ A |\Upsilon_t|_{\omega_t} \rd \left( \log|\Upsilon_t|_{\omega_t} \right) \wedge \omega_t.} 
    \end{align}
    The other equation is obtained by substituting the above into \eqref{Eq:beta.t.gamma.t.cCY-modified}.
\end{proof}

\section{Solutions from a particular initial condition} \label{Sect:Modified.coflow.LSES}

We now study a particular solution of the modified Laplacian coflow \eqref{Eq:Modified.coflow} analogous to that obtained in  \cite{Lotay2022}.

Let $(M^7, \eta, \Phi, \Upsilon)$ be a contact Calabi--Yau $7$-manifold and suppose that $(\omega, \Upsilon)$ is a transverse Calabi--Yau structure, that is, $\omega = \rd \eta$ is transverse Ricci-flat and $\Upsilon$ is a nowhere-vanishing transverse holomorphic $(3,0)$-form. Recall that in this case, the norm $|\Upsilon|_\omega$ is constant and can be scaled to be $1$. Consider a family of $\rG_2$-structures on $M$ defined by
\begin{align} 
\label{Eq:varphi_t.LSES}
    \varphi_t = b_t^3 \Re \Upsilon + a_t b_t^2 \eta \wedge \omega 
\end{align}
where the functions $a_t$, $b_t$ depend only on $t$ and are constant on $M$. The induced metrics $g_t$ and volume forms $\vol_t$ can be checked to be
\begin{align}\label{eq: metric.vol.t}
    g_t = a_t^2 \eta^2 + b_t^2 g|_{\cD} 
    \qandq
    \vol_t = a_t b_t^6 \eta \wedge \vol|_{\cD}.
\end{align}
It follows that the dual $4$-form $\psi_t$ is
\begin{align}
    \psi_t = - a_t b_t^3 \eta \wedge \Im \Upsilon + \frac{1}{2} b_t^4 \omega^2.
\end{align}

We set initial conditions for the fibre radius $a_0 = \epsilon$ and basic scale $b_0 = 1$, so that
\begin{align}
\label{eq: phi_0}
    \varphi_0 = \Re \Upsilon + \epsilon \eta \wedge \omega
    \qandq
    \psi_0 = - \epsilon \eta \wedge \Im \Upsilon + \frac{1}{2} \omega^2,
\end{align}
We have the following expressions for exterior derivatives and torsion forms along the family $\{\varphi_t\}$.

\begin{lemma}[\cite{Lotay2022} {Propositions 4.5 and 4.6}]
\label{Lem:torsion.ccY.LSES}
    Let $\varphi_t$ be defined by \eqref{Eq:varphi_t.LSES}, then
    \begin{align}
        \rd \varphi_t = a_t b_t^2 \omega^2, \qquad \rd \psi_t = 0, \qquad \ast_t \rd \varphi_t = 2 a_t^2 \eta \wedge \omega.
    \end{align}
    As such, the torsion forms are
    \begin{align}
        (\tau_0)_t = \frac{6 a_t}{7 b_t^2}, \qquad (\tau_1)_t = 0, \qquad (\tau_2)_t = 0, \qquad (\tau_3)_t = - \frac{6}{7} a_t b_t \Re \Upsilon + \frac{8}{7} a_t^2 \eta \wedge \omega .
    \end{align}
    Moreover,  the full torsion $2$-tensor $T_t$ is given by    \begin{equation}
    \label{Eq.Torsion}
        T_t=-\frac{3}{2}a_t^3b_t^{-2}\eta^2+\frac{1}{2}a_tg|_{\cD}=-2a_t^3b_t^{-2}\eta^2+\frac{1}{2}a_tb_t^{-2}g_t
    \end{equation}
and it has the following derived quantities: 
\begin{align}
    |T_t|_{g_t}^2=\frac{15}{4}a_t^2b_t^{-4}, \quad \div T_t=0, \quad |\nabla_t T_t|_{g_t}^2&=c_0a_t^4b_t^{-8},
\end{align}
for some constant $c_0\in\R$.
\end{lemma}

\subsection{Solving the modified Laplacian coflow}    
We now proceed in a similar way to \cite{Lotay2022}, obtaining an ODE in terms of $a_t$ and $b_t$ such that the family $\{\varphi_t\}$ satisfies the modified coflow.

\begin{theorem} \label{Thm:modified.coflow.LSES}
    The family of $\mathrm{G}_2$-structures $\{\varphi_t\}$ defined by \eqref{Eq:varphi_t.LSES} solves the modified Laplacian coflow with initial condition \eqref{eq: phi_0} if, and only if, the functions $a_t$ and $b_t$ satisfy:
\begin{align}
    a_t &= \epsilon b_t^{-3}, \label{Eq:a_t.b_t.LSES}\\
    \frac{d}{d t} b_t &= \frac{1}{2} \epsilon b_t^{-9} (A b_t^5 - \epsilon),
\label{Eq:a_t.b_t.2.LSES}
\end{align}
    with $a_0 = \epsilon$ and $b_0 = 1$. 
\end{theorem}
\begin{proof}
    One can readily check that

    \begin{align}
        \Delta_t\psi_t = \rd \ast_t \rd \varphi_t = \rd (2a_t^2 \eta \wedge \omega) = 2 a_t^2 \omega^2.
    \end{align}
    Additionally from Lemma \ref{Lem:torsion.ccY.LSES}, we have
    \begin{align}
        \rd \left( \left( A - \frac{7}{2} (\tau_0)_t \right) \varphi_t \right) = \left( A - 3 \frac{a_t}{b_t^2} \right) a_t b_t^2 \omega^2 = a_t (A b_t^2 - 3 a_t) \omega^2.
    \end{align}

    Assuming our Ansatz along the flow, we have
    \begin{align}
        \frac{\del}{\del t} \psi_t = \frac{1}{2} \frac{d}{dt} (b_t^4) \omega^2 - \frac{d}{dt} (a_t b_t^3) \eta \wedge \Im \Upsilon,
    \end{align}
    hence the modified coflow results in the evolution equations
    \begin{align}
        \frac{d}{dt} (b_t^4) = 2 a_t (A b_t^2 - a_t) 
        \qandq 
        \frac{d}{dt} (a_t b_t^3) = 0.
    \end{align}

    The latter equation and initial conditions imply that $a_t = \epsilon b_t^{-3}$. Plugging this back into the former equation yields
    \begin{align}
        \frac{d}{dt} (b_t^4) = 2 \epsilon b_t^{-6} (A b_t^5 - \epsilon)
    \end{align}
    or, equivalently,
    \begin{align}
        \frac{d}{dt} b_t = \frac{1}{2} \epsilon b_t^{-9} (A b_t^5 - \epsilon),
    \end{align}
    as claimed.
\end{proof}

In particular, when $A=0$, substituting $a_t = \epsilon b_t^{-3}$  into \eqref{Eq:a_t.b_t.2.LSES} gives:  

\begin{cor}
\label{Cor: modified.coflow.A=0.LSES}
    Consider the functions 
\begin{align}
    a_t 
    &= \epsilon (1 - 5\epsilon^2 t)^{-\frac{3}{10}}, \\
    b_t 
    &= (1 - 5\epsilon^2 t)^{\frac{1}{10}}.
\end{align}
    Then the family $\{\varphi_t\}$ of $\rG_2$-structures defined by \eqref{Eq:varphi_t.LSES} satisfies the modified Laplacian coflow with constant $A = 0$.
\end{cor}

\subsection{Singularity analysis}
\label{Subsect:Sing}

Since $a_t$ is an explicit function of $b_t$, then $\varphi_t$ depends only on $b_t$. Consequently, the Riemannian tensors relevant to the modified Laplacian coflow are derived from and measured by the $\gt$-metric \eqref{eq: metric.vol.t} induced by $\varphi_t$. Therefore, we are particularly interested in the behavior of the system:

\begin{align} \label{Eq:ODE}
    \frac{d}{dt} b_t 
    = \frac{1}{2} \epsilon b_t^{-9} (A b_t^5 - \epsilon).
\end{align}

Understanding the dynamics of $b_t$ will provide insight into the evolution of the $\gt$-structure and the associated geometric quantities under the modified Laplacian coflow. 

\begin{table}[H]
\centering
\begin{tabular}{|c|c|c|}
\hline
\textbf{Condition} & \textbf{Steady State} & \textbf{Solution Behavior} \\
\hline
$A < 0$ & $(\frac{\epsilon}{A})^{\frac{1}{5}} < 0$ (stable) & $b_t$ with $b_0 = 1$ monotonically decreasing \\
\hline
$A = 0$ & No steady state & \begin{tabular}{@{}c@{}} $b_t = (1-5\epsilon^2 t)^{\frac{1}{10}}$, \\ monotonically decreasing and \\ collapses at $T = \frac{1}{5\epsilon^2}$ \end{tabular} \\
\hline
$0 < A < \epsilon$ & $0 < (\frac{\epsilon}{A})^{\frac{1}{5}} < 1$ (unstable) & $b_t$ with $b_0 = 1$ monotonically decreasing \\
\hline
$A = \epsilon > 0$ & $1$ (unstable) & $b_t$ with $b_0 = 1$ is constant \\
\hline
$0 < \epsilon < A$ & $(\frac{\epsilon}{A})^{\frac{1}{5}} > 1$ (unstable) & $b_t$ with $b_0 = 1$ monotonically increasing \\
\hline
\end{tabular}
\caption{Summary of steady state and solution behaviour for various regimes of $A$ and $\epsilon$.}
\label{table:summary}
\end{table}

The ODE \eqref{Eq:ODE} is separable, and it can be checked that if $A \neq 0, \eps$, then the solution with $b_0 = 1$ satisfies
\begin{equation}
    \frac{b^5}{5A} + \frac{\epsilon}{5A^2} \ln|Ab^5-\epsilon| = \frac{1}{2} \epsilon t + \frac{1}{5A} + \frac{\epsilon}{5A^2} \ln|A -\epsilon|.
\end{equation}

If $A < \epsilon$, then the solution $b \rightarrow 0$ as
\begin{equation}
    t \rightarrow \frac{2}{5\epsilon A} \left[ \frac{\epsilon}{A} \ln \left| \frac{\epsilon}{\epsilon - A} \right| - 1 \right],
\end{equation}
and $A > \epsilon$ then the solution $b \rightarrow \infty$ as $t \rightarrow \infty$.

Following the approach in \cite{Lotay2022}, we can use a similar quantity to characterize the formation of finite-time singularities when $A = 0$ using the explicit expression for $b_t$. Define 
\begin{equation*}
    \Lambda(x,t)=(|Rm(x,t)|^2_{g_t} +|T(x,t)|^4_{g_t} +|\nabla^{g_t} T(x,t)|^2_{g_t})^{\frac{1}{2}}
\end{equation*}
for $x\in M$ and time $t$. We then let
\begin{equation}\label{eq:Lambda.t}
    \Lambda(t)=\sup_{x\in M}\Lambda(x,t)
\end{equation}
As a direct consequence of \cite{Lotay2022}*{Proposition 3.5}, we find that the norm of the Riemann curvature tensor associated with the metric $g_t$, related to the solution of the modified Laplacian coflow $\varphi_t$, is given by
\begin{align}
 |Rm_t|^2_{g_t}= b_t^{-4} |Rm_0^{\cD_0}|_{g_0}^2 + b_t^{-20} c_0\epsilon^4. 
\end{align}
We can plug in the quantities from Lemma \ref{Lem:torsion.ccY.LSES} to compute $\Lambda$ of the family of $\mathrm{G}_2$-structures $\{\varphi_t\}$ defined by \eqref{Eq:varphi_t.LSES}, which solves the modified Laplacian coflow with initial condition \eqref{eq: phi_0} and satisfies the system \eqref{Eq:a_t.b_t.LSES} and \eqref{Eq:a_t.b_t.2.LSES}. In particular, we get

\begin{equation*}
    \Lambda(x,t)=b_t^{-10} \left(b_t^{16}|Rm(x)_0^{\cD_0}|^2_{g_0} +2c_0\epsilon^4+\left(\frac{15}{4}\right)^2\epsilon^4 \right)^{\frac{1}{2}}
\end{equation*}
{In \cite{Chen2018}, Chen defined a class of \emph{reasonable flows} of $G_2$-structures, established a Shi-type estimate, and used it to derive an estimate for the blow-up rate on a compact manifold. Moreover, the modified Laplacian coflow is included in this class of flows. (The Laplacian coflow is not included in this set since it is yet to be shown if it has short-time existence and uniqueness.) We therefore introduce the following definition, which will be useful in the analysis of singularities.}
\begin{definition}\label{dfn:sing.types}
    Suppose that $(M^7,\varphi_t,\psi_t,g_t)$ is a solution to a {reasonable} flow of $\gt$-structures on a closed manifold on a maximal time interval $[0,T)$ and let $\Lambda(t)$ be as in \eqref{eq:Lambda.t}.
    
    If we have a finite-time singularity, i.e.~$T<\infty$, we say that the solution forms  
\begin{itemize}
    \item a \emph{Type I singularity} (rapidly forming) if $\sup_{t\in[0,T)}(T-t)\Lambda(t) <\infty$; and otherwise 
    \item a \emph{Type IIa singularity} (slowly forming) if $\sup_{t\in [0,T)}(T-t)\Lambda(t)=\infty$.
\end{itemize}
\end{definition}

If $A = 0$, as indicated by Corollary \ref{Cor: modified.coflow.A=0.LSES}, the solutions take the form $a_t = \epsilon (1 - 5\epsilon^2 t)^{-3/10}$ and $b_t = (1 - 5\epsilon^2 t)^{1/10}$. In this context, we can analyze the asymptotic behavior of the solutions of the modified Laplacian coflow, akin to the approach outlined in \cite{Lotay2022}. Specifically, we explore how the solutions behave as $t$ approaches the maximal time $\frac{1}{5\epsilon^2}$, drawing parallels to the conclusions drawn in the study of the Laplacian flow in \cite{Lotay2022}.

\begin{prop}
    Let $(M^7, \eta, \Phi, \Upsilon)$ be a compact contact Calabi--Yau $7$-manifold with transverse Ricci-flat K\"{a}hler form $\omega = d\eta$ and transverse holomorphic $(3,0)$-form $\Upsilon$.
    The solution to the modified Laplacian coflow with $A = 0$ and initial condition \eqref{eq: phi_0} has a Type I finite-time singularity at $T = \frac{1}{5\epsilon^2}$. Further, after normalising $(M,g_t)$ to a fixed volume, the solution collapses to $\mathbb{R}$, as $t \to T$.
\end{prop}

\section{Breaking the Sasakian structure on a \texorpdfstring{$cCY^7$}{cCY7}} \label{Sect:Breaking.Sasakian}

We now revisit the setup from \S \ref{Sect:cCY}, on a contact Calabi--Yau $7$-manifold $(M^7, \eta, \Phi, \Upsilon)$. Recall that we considered deformations of type II, given by a $1$-parameter family of basic functions $\{f_t\}$, which determines at each $t$ the contact form $\eta_t$, and transverse K\"{a}hler form $\omega_t$ by
\begin{align}
    \eta_t &= \eta + \rd^c f_t, \\
    \omega_t &= d\eta + \rd \rd^c f_t.
\end{align}
Using ideas from \cite{Picard2022flows} and the transverse $\partial \overline{\partial}$-lemma of \cite{Kacimi1990} (see Appendix \ref{Ap.Sasakian}), we now allow the transverse K\"{a}hler structure to vary within the basic cohomology class $[\rd \eta]_B$. This added freedom does not change the transverse complex structure $J$, and so $\Upsilon$ remains a transverse holomorphic volume form throughout.

In other words, we consider on $(M, \eta, \Phi, \Upsilon)$  a transverse $\SU(3)$-structure $(\omega', \Upsilon)$, where $\omega' \in [\rd \eta]_B$. By El Kacimi-Alaoui's transverse $\partial \overline{\partial}$-lemma, we can write
\begin{align}
    \omega' = \rd \eta + \rd \rd^c h,
\end{align}
where $h$ is a basic function. Note that we are determining the function $h$ (up to addition of a constant) from our choice of $\omega'$ and not vice versa. In some sense, we can consider this a breaking of the Sasakian structure, since the transverse K\"{a}hler form $\omega'$ is no longer determined by the contact form $\eta$.

In a similar manner to \S \ref{Sect:cCY}, we define a $\rG_2$-structure by
\begin{align}
    \varphi = \Re \left( \frac{1}{|\Upsilon|_{\omega'}} \Upsilon \right) + |\Upsilon|_{\omega'} \eta \wedge \omega'. \label{Eq:varphi.cCY.moving}
\end{align}
One can verify that the induced metric and volume form on $M$ are
\begin{align}
    g = |\Upsilon|_{\omega'}^2 \eta^2 + g'|_{\cD},
\qandq
    \vol = |\Upsilon|_{\omega'} \eta \wedge \vol'|_{\cD}.
\end{align}
Furthermore, the Hodge star operator acts on a basic $k$-form $\alpha$ by
\begin{align}
    \ast \alpha &= (-1)^k |\Upsilon|_{\omega'} (\eta \wedge \ast_B \alpha), \\
    \ast (\eta \wedge \alpha) &= \frac{1}{|\Upsilon|_{\omega'}} \ast \alpha.
\end{align}
Hence the dual $4$-form is
\begin{align}
    \psi = \ast\varphi
    =\eta \wedge \Im \Upsilon + \frac{1}{2} \omega'^2.
\end{align}
It is easy to see that $\rd \psi = 0$ and so $\varphi$ is a coclosed $\rG_2$-structure.
As in \S \ref{Sect:cCY}, we compute the torsion forms and the Hodge Laplacian of this $\rG_2$-structure.

\begin{prop} \label{Prop:Torsion.cCY.moving}
    Let $(M, \eta, \Phi, \Upsilon)$ be a contact Calabi--Yau $7$-manifold. Let $\omega' \in [\rd \eta]_B$ be a transverse K\"{a}hler structure and $\varphi$ be the $\rG_2$-structure defined by \eqref{Eq:varphi.cCY.moving}. Then the torsion forms of $\varphi$ are given by
    \begin{align}
        \tau_0 = \frac{6}{7} |\Upsilon|_{\omega'}, \qquad \tau_1 = 0, \qquad \tau_2 = 0,
    \end{align}
    and
    \begin{align}
        \tau_3 = (\nabla \log |\Upsilon|_{\omega'}) \lrcorner \left( -\eta \wedge \Im \Upsilon - \frac{1}{2} \omega'^2 \right) - \frac{6}{7} \Re \Upsilon + \frac{8}{7} |\Upsilon|_{\omega'}^2 \eta \wedge d\eta.
    \end{align}
\end{prop}

The proof is similar to that of Proposition \ref{Prop:Torsion.ccY}, however K\"{a}hler identities are invoked to deal with the extra terms. Analogous methods yield the Hodge Laplacian:

\begin{prop} \label{Prop:Laplacian.psi.cCY.moving}
    Let $(M, \eta, \Phi, \Upsilon)$ be a contact Calabi--Yau $7$-manifold Let $\omega' \in [\rd \eta]_B$ be a transverse K\"{a}hler structure and $\varphi$ be the $\rG_2$-structure defined by \eqref{Eq:varphi.cCY.moving}. Then we have
    \begin{align}
        \Delta \psi &= \cL_{\nabla (\log |\Upsilon|_{\omega'})} \left( - \eta \wedge \Im \Upsilon - \frac{1}{2} \omega'^2 \right) \\ \nonumber
        &+ 2 |\Upsilon|_{\omega'}^2 \rd (\log |\Upsilon|_{\omega'}) \wedge \eta  \wedge \left[ 3\omega' - \rd \eta \right] + |\Upsilon|_{\omega'}^2 d\eta \wedge \left[ 3\omega' - \rd \eta \right].
    \end{align}
\end{prop}

\subsection{The Laplacian coflow} \label{Subsect:Laplacian.coflow.cCY.moving}

We now apply the Laplacian coflow equation to a family of such structures and consider Ans\"{a}tze on our choices of $\eta_t$, $\omega'_t$ and $\Upsilon_t$. We recall that our construction of $\rG_2$-structures previously required certain compatibility conditions to hold.

\begin{theorem} \label{Thm:Laplacian.coflow.ccY.moving}
    Let $(M^7, \eta, \Phi, \Upsilon)$ be a contact Calabi--Yau $7$-manifold and let $\omega' \in [\rd \eta]_B$ be a transverse K\"{a}hler form. Suppose we have a family of contact forms $\eta_t$ and a family of compatible transverse $\SU(3)$-structure $(\omega'_t, \Upsilon_t)$ on $M$ with initial conditions $\eta_0 = \eta$, $\omega'_0 = \omega'$ and $\Upsilon_0 = \Upsilon$ satisfying the coupled differential equations:
    \begin{align}
        \frac{\del}{\del t} \eta_t 
        &= \cL_{\nabla_t (\log |\Upsilon_t|_{\omega'_t})} \eta_t + \alpha_t, \label{Eq:eta.t.system.cCY.moving} \\
        \frac{\del}{\del t} \omega'_t 
        &= - \cL_{\nabla_t (\log |\Upsilon_t|_{\omega'_t})} \omega'_t + \beta_t, \label{Eq:omega.t.system.cCY.moving} \\
        \frac{\del}{\del t} \Upsilon_t 
        &= \cL_{\nabla_t (\log |\Upsilon_t|_{\omega'_t})} \Upsilon_t + \gamma_t. \label{Eq:Upsilon.t.system.cCY.moving}
    \end{align}

    Then the family of $\rG_2$-structures given by
    \begin{align}
        \varphi_t 
        = \Re \left( \frac{1}{|\Upsilon|_{\omega'}} \Upsilon \right) + |\Upsilon|_{\omega'} \eta_t \wedge \omega'_t,
    \end{align}
    is a solution to the Laplacian coflow \eqref{eq: Laplacian.coflow} if, and only if,
    \begin{align}
        \alpha_t &= 0, \\
        -\eta_t \wedge \gamma_t - \omega'_t \wedge \beta_t &= 2 |\Upsilon|_{\omega'} \rd (\log |\Upsilon|_{\omega'}) \wedge \eta_t \wedge \left[ 3\omega'_t - \rd \eta_t \right] + |\Upsilon|_{\omega'}^2 \rd \eta \wedge \left[ 3\omega'_t - \rd \eta_t \right].
    \end{align}
\end{theorem}

\begin{proof}
    We see that the associated $4$-form is given by
    \begin{align}
        \psi_t=- \eta_t \wedge \Im \Upsilon_t + \frac{1}{2} (\omega'_t)^2.
    \end{align}
    Differentiating with respect to time, we get
    \[
        \frac{\del}{\del t} \psi_t = -\eta_t \wedge \left( \frac{\del}{\del t} \Im \Upsilon_t \right) - \frac{1}{2} \left( \frac{\del}{\del t} (\omega'_t)^2 \right) - \left( \frac{\del}{\del t} \eta_t \right) \wedge \Im \Upsilon_t.
    \]
    Applying Proposition \ref{Prop:Laplacian.psi.cCY.moving} we get
    \begin{align*}
         &-\eta_t \wedge \left( \frac{\del}{\del t} \Im \Upsilon_t \right) - \left( \frac{\del}{\del t} \eta_t \right) \wedge \Im \Upsilon_t - \frac{1}{2} \left( \frac{\del}{\del t} (\omega'_t)^2 \right) \\
        &= \cL_{\nabla_t (\log |\Upsilon_t|_{\omega'_t})} \left( - \eta_t \wedge \Im \Upsilon_t - \frac{1}{2} (\omega'_t)^2 \right)\\ 
        &\quad + 2 |\Upsilon_t|_{\omega'_t}^2 \rd (\log |\Upsilon_t|_{\omega'_t}) \wedge \eta_t  \wedge \left[ 3\omega'_t - \rd \eta_t \right] + |\Upsilon_t|_{\omega'_t}^2 \rd \eta_t \wedge \left[ 3\omega'_t - \rd \eta_t \right].
    \end{align*}
    Substituting in the systems \eqref{Eq:eta.t.system.cCY.moving}--\eqref{Eq:Upsilon.t.system.cCY.moving}, we obtain
    \[
        -\eta_t \wedge \gamma_t - \alpha_t \wedge \Im \Upsilon_t - \omega'_t \wedge \beta_t = 2 |\Upsilon_t|_{\omega'_t}^2 \rd (\log |\Upsilon_t|_{\omega'_t}) \wedge \eta_t  \wedge \left[ 3\omega'_t - \rd \eta_t \right] + |\Upsilon_t|_{\omega'_t}^2 \rd \eta_t \wedge \left[ 3\omega'_t - d\eta_t \right].
    \]

    We now consider the type decomposition of each term in the above expression with respect to the transverse complex structure $J_t$. On the RHS, the first term
    is of type $\eta_t \wedge [(1,2) \oplus (2,1)]$ and the second term is of type $(2,2)$. On the other hand, term $\Im \Upsilon_t$ is of type $(0,3) \oplus (3,0)$ and $\omega'_t$ is of type $(1,1)$.
    We conclude that $\alpha_t = 0$, and we obtain the desired expression relating $\beta_t$ and $\gamma_t$. In addition, we see that $\beta_t$ is of type $\eta_t \wedge [(0,1) \oplus (1,0)] \oplus (1,1)$ and $\gamma_t$ is of type $(1,2) \oplus (2,1)$.
\end{proof}

\subsection{The modified Laplacian coflow} \label{Subsect:Modified.coflow.cCY.moving}

We perform a similar analysis of the modified Laplacian coflow as that of \S \ref{Subsect:Modified.cof.ccY}, for this new Ansatz. Once again, we note that the torsion form $\tau_0 = \frac{6}{7} |\Upsilon|_{\omega'}$, hence the extra terms with constant $A$ are given by
\begin{align}
    \rd \left( \left( A - \frac{7}{2} \tau_0 \right) \varphi \right) &= - \frac{A}{|\Upsilon|_{\omega'}} d(\log |\Upsilon|_{\omega'}) \wedge \Re (\Upsilon) + A |\Upsilon|_{\omega'} d(\log |\Upsilon|_{\omega'}) \wedge \eta \wedge \omega' + A |\Upsilon|_{\omega'} d\eta \wedge \omega' \nonumber \\
    &\qquad - 6 |\Upsilon|^2_{\omega'} d(\log |\Upsilon|_{\omega'}) \wedge \eta \wedge \omega' - 3 |\Upsilon|^2_{\omega'} d\eta \wedge \omega'. \label{Eq:Modified.coflow.extra.cCY.moving}
\end{align}

Taking into account these extra terms, we get the analogous result for the modified coflow.

\begin{theorem} \label{Thm:Modified.coflow.ccY.moving}
    Let $(M^7, \eta, \Phi, \Upsilon)$ be a contact Calabi--Yau $7$-manifold and let $\omega' \in [d\eta]_B$ be a transverse K\"{a}hler form. Suppose we have a family of contact forms $\{\eta_t\}$ and a family of compatible transverse $\SU(3)$-structure $(\omega'_t, \Upsilon_t)$ on $M$, with initial conditions $\eta_0 = \eta$, $\omega'_0 = \omega'$ and $\Upsilon_0 = \Upsilon$,  satisfying the coupled differential equations
    \begin{align}
        \frac{\del}{\del t} \eta_t = \cL_{\nabla_t (\log |\Upsilon_t|_{\omega'_t})} \eta_t + \alpha_t, \label{Eq:eta.t.system.cCY.moving-2} \\
        \frac{\del}{\del t} \omega'_t = - \cL_{\nabla_t (\log |\Upsilon_t|_{\omega'_t})} \omega'_t + \beta_t, \label{Eq:omega.t.system.cCY.moving-2} \\
        \frac{\del}{\del t} \Upsilon_t = \cL_{\nabla_t (\log |\Upsilon_t|_{\omega'_t})} \Upsilon_t + \gamma_t. \label{Eq:Upsilon.t.system.cCY.moving-2}
    \end{align}

    Then the family of $\rG_2$-structures given by
    \begin{align}
        \varphi_t = \Re \left( \frac{1}{|\Upsilon|_{\omega'}} \Upsilon \right) + |\Upsilon|_{\omega'} \eta_t \wedge \omega'_t,
    \end{align}
    is a solution to the modified coflow \eqref{Eq:Modified.coflow} if, and only if, 
    \begin{align}
        \alpha_t &= \mathcolor{blue}{\frac{A}{|\Upsilon_t|_{\omega_t'}} \rd ( \log |\Upsilon_t|_{\omega_t'}),} \\ \nonumber
        -\eta_t \wedge \gamma_t - \omega'_t \wedge \beta_t &= 2 |\Upsilon|_{\omega'} \rd (\log |\Upsilon|_{\omega'}) \wedge \eta_t \wedge \left[ 3\omega'_t - \rd \eta_t \right] + |\Upsilon|_{\omega'}^2 \rd \eta \wedge \left[ 3\omega'_t - \rd \eta_t \right] \\
        & \mathcolor{blue}{+A |\Upsilon_t|_{\omega_t'} \rd (\log |\Upsilon_t|_{\omega_t'}) \wedge \eta_t \wedge \omega_t' + A |\Upsilon_t|_{\omega_t'} \rd \eta_t \wedge \omega_t'} \\ \nonumber
        & \mathcolor{blue}{-6 |\Upsilon_t|^2_{\omega_t'} \rd (\log |\Upsilon_t|_{\omega_t'}) \wedge \eta_t \wedge \omega_t' -3 |\Upsilon_t|^2_{\omega_t'} \rd \eta_t \wedge \omega_t'.}
    \end{align}
\end{theorem}

\subsection{Possible Further Directions} \label{Sect:Possible.solutions}

We speculate how to obtain solutions to the Laplacian coflows from this setup. To do this, we continue from the previous section and follow \cite{Picard2022flows}, by considering pullback via a family of diffeomorphisms. Suppose $\wtilde{\omega_t}'$ is the solution to some perturbed Sasaki--Ricci (transverse K\"{a}hler--Ricci) flow (see \cite{SWZ})
\begin{align}
    \frac{\del}{\del t} \wtilde{\omega_t}' = -2 \Ric^T (\wtilde{\omega_t}',J) + \wtilde{\beta_t},
\end{align}
where $\Ric^T$ denotes the transverse Ricci form. We can use these transverse K\"{a}hler forms to define a time-dependent vector field
\begin{align}
    Y_t = \wtilde{\nabla_t}' (\log |\Upsilon|_{\wtilde{\omega_t}'}),
\end{align}
where $\wtilde{\nabla_t}'$ denotes the Levi-Civita connection of $\wtilde{\omega_t}'$. In turn, we can use this time-dependent vector field to obtain a family of diffeomorphisms $\Theta_t$ satisfying
\begin{align}
    \frac{\del}{\del t} \Theta_t (p) = Y_t(p), \qquad \Theta_0 = id.
\end{align}

Suppose further that $\wtilde{\eta_t}$ and $\wtilde{\Upsilon_t}$ are flows of contact forms and transverse holomorphic volume forms, respectively,  satisfying appropriate compatibility conditions:
\begin{itemize}
    \item $(\wtilde{\omega_t}', \wtilde{\Upsilon_t})$ is a transverse $\SU(3)$-structure with respect to $\wtilde{\eta_t}$, and
    \item $\wtilde{\omega_t}' \in [\rd \wtilde{\eta_t}]_B$.
\end{itemize}
Writing
$\frac{\del}{\del t} \wtilde{\eta_t} 
    = \wtilde{\alpha_t}$ and $\frac{\del}{\del t} \wtilde{\Upsilon_t} 
    = \wtilde{\gamma_t}$,
and pulling back by the diffeomorphisms $\Theta_t$, we can define structures
$$
     \eta_t = \Theta_t^* \wtilde{\eta_t}, \qquad \omega_t' = \Theta_t^* \wtilde{\omega_t}', \qquad \Upsilon_t = \Theta_t^* \wtilde{\Upsilon_t}.
$$
A computation shows that
\begin{align*}
    \frac{\del}{\del t} \eta_t &= \cL_{\nabla_t (\log |\Upsilon_t|_{\omega'_t})} \eta_t + \Theta_t^* \wtilde{\alpha_t}, \\
    \frac{\del}{\del t} \omega'_t 
    &= - \cL_{\nabla_t (\log |\Upsilon_t|_{\omega'_t})} \omega'_t + \Theta_t^* \wtilde{\beta_t}, \\
    \frac{\del}{\del t} \Upsilon_t 
    &= \cL_{\nabla_t (\log |\Upsilon_t|_{\omega'_t})} \Upsilon_t + \Theta_t^* \wtilde{\gamma_t}.
\end{align*}
We thus see that if the auxiliary flows {can be} chosen appropriately, then the pullback yields solutions to the Laplacian coflows.

\begin{remark} \label{Rem:A.priori}
    {It is still unclear if there exist such auxiliary flows that satisfy the above equations. One particular difficulty in finding these is because $\wtilde{\Upsilon_t}$ can only vary by phase shifts since the transverse complex structure is fixed along the Sasaki--Ricci flow.}
\end{remark}

\begin{remark} \label{Rem:More.general.choices}
    This method of obtaining potential solutions also does not encompass the solutions of the Laplacian coflow in \cite{Lotay2022} and those of the modified coflow discuss in \S \ref{Sect:Modified.coflow.LSES}. This is because the functions $a_t$ and $b_t$ scale the transverse K\"{a}hler class. One can instead include similar scaling functions $a_t$ and $b_t$ depending on time to match those solutions, however these introduce more freedoms in how the parameters interact with one another.

    The solutions in \cite{Lotay2022} and \S \ref{Sect:Modified.coflow.LSES} take advantage of starting with a transverse Ricci-flat K\"{a}hler structure, which greatly simplifies the evolution equations, as $|\Upsilon_t|_{\omega_t}$ is just a constant in that case. Altogether, these two methods suggest that a more general transverse flow should be considered, where $\wtilde{\omega_t}'$ can be allowed to move freely through the transverse K\"{a}hler cone.
\end{remark}

This pullback idea can be applied to the earlier cases in \S \ref{Sect:Product} and in \S \ref{Sect:cCY} by adapting the equations accordingly. In those cases, we keep the $S^1$-invariance along the flow, and thus cannot even optimistically expect to obtain torsion-free $\rG_2$ metrics with holonomy $\rG_2$.

\appendix

\section{Sasakian manifolds} \label{Ap.Sasakian}

We briefly review Sasakian manifolds and discuss some useful results involving deformations of Sasakian structures. These occur at the beginning of \S \ref{Sect:cCY} and \S\ref{Sect:Breaking.Sasakian},  defining certain families of $\rG_2$-structures on contact Calabi--Yau $7$-manifolds.

\begin{definition} \label{Def:Contact.structure}
    A contact structure on a $(2n+1)$-manifold $M^{2n+1}$ is a triple $(\xi, \eta, \Phi)$ where $\xi$ is a vector field (called the Reeb vector field), $\eta$ is a $1$-form (called the contact form), and $\Phi$ is a $(1,1)$-tensor such that
    \begin{align}
        \eta(\xi) = 1, \qquad \Phi^2 = -1 + \xi \otimes \eta, \label{Eq:Almost.contact.structure}
    \end{align}
    and
    \begin{align}
        \eta \wedge (\rd \eta)^n \neq 0. \label{Eq:Non-degen}
    \end{align}
\end{definition}

Using the Reeb vector field $\xi$, we obtain a $1$-foliation $\cF_\xi$, and its dual $1$-form $\eta$ determines a codimension $1$ subbundle $\cD = \ker \eta$ of $TM$. We have a canonical splitting
\begin{align}
    TM = \cD \oplus L \xi, \label{Eq:Splitting}
\end{align}
where $L \xi$ is the line bundle spanned by $\xi$. 
The second condition in \eqref{Eq:Almost.contact.structure} implies that the restriction of $\Phi$ to $\cD$ results in an almost-complex structure $J = \Phi|_{\cD}$. We can also consider the quotient bundle $\nu(\cF_\xi) = TM / L\xi$ of the canonical foliation $\cF_\xi$. This space can be identified with $\cD$, however it is convenient to distinguish them, as we aim to deform Sasakian structures by varying one, while keeping the other one fixed.

A Riemannian metric $g$ on $M$ is compatible with the contact structure if
\begin{align}
    g(\Phi(X), \Phi(Y)) = g(X,Y) - \eta(X) \eta(Y),
\end{align}
for any vector fields $X,Y$ on $M$. Such a metric induces an almost-Hermitian metric on $\cD$ and makes the decomposition in \eqref{Eq:Splitting} orthogonal. In this case, the quadruple $(\xi, \eta, \Phi, g)$ is called a contact metric structure. If the metric cone $(C(M),\overline{g}) = (\R_{> 0} \times M, dr^2 + r^2g)$ is K\"{a}hler, then we call the quadruple $(\xi, \eta, \Phi, g)$ a Sasakian structure. Since the Reeb vector field $\xi$ of a Sasakian structure defines several important spaces and bundles, we will define some properties related to basic $k$-forms. 

\begin{definition} \label{Def:Basic.form}
    A $k$-form $\alpha$ on a contact manifold is called \emph{basic} if
    \begin{align}
        \xi \lrcorner\, \alpha = 0, \qquad \cL_\xi \alpha = 0. \label{Eq:Basic.form}
    \end{align}
\end{definition}

Using Cartan's magic formula, one can see that the Lie derivative condition is equivalent to $\xi \lrcorner\, (d\alpha) = 0$, and so the exterior derivative preserves basic forms. Basic cohomology classes, denoted by $[ \cdot ]_B$, can be defined in the usual way with the appropriate restrictions. 

Given a Sasakian structure $\cS = (\xi, \eta, \Phi, g)$ on $M^{2n+1}$, we wish to deform it and obtain Sasakian structures that preserve the {Reeb} vector field $\xi$. We denote this set by
\begin{align}
    \fF(\xi)=\{\textrm{Sasakian structures }\, \cS'=(\xi', \eta', \Phi', g'): \xi^{'}=\xi\}.
\end{align}
Given two Sasakian structures $\cS, \cS' \in \fF(\xi)$ with contact forms $\eta$ and $\eta'$ respectively, one can check that $\zeta = \eta - \eta^{'}$ is basic. As such $[\rd\eta^{'}]_{B}=[\rd\eta]_{B}$ and hence all Sasakian structures in $\fF(\xi)$ correspond to the same basic cohomology class.

Let $\overline{J}$ denote the induced complex structure on $\nu(\cF_\xi)$ and let $\pi_\nu \: TM \rightarrow \nu(\cF_\xi)$ be the quotient map. We define the subset $\fF(\xi, \overline{J}) \subseteq \fF(\xi)$ to be the subset of all Sasakian structures $(\xi^{'},\eta^{'},\Phi^{'},g^{'})\in\fF(\xi)$ such that the diagram
\[
\begin{CD}
TM@>{\Phi'}>>TM\\
@V{\pi_\nu}VV @VV{\pi_\nu}V\\
Q@>>{\overline{J}}>Q
\end{CD}
\]
commutes. The elements of $\fF(\xi,\overline{J})$ are the Sasakian structures with the same transverse holomorphic structure $\overline{J}$. 

We can now give an alternative description of $\fF(\xi, \overline{J})$, but first we need the transverse $ \partial\overline{\partial}$ Lemma due to El Kacimi-Alaoui:

\begin{prop}[\cite{Kacimi1990} {Proposition 3.5.1}] \label{Prop:Transverse.del.delbar}
    Let $(M,\cS)$ be a compact Sasakian manifold, and $\omega,\omega^{'}$ be basic real closed $(1,1)$-forms such that $[\omega]_{B}=[\omega^{'}]_B$. Then there exists a smooth basic function $h$ such that
    \begin{align}
        \omega' = \omega + \sqrt{-1} \partial \overline{\partial} \phi = \omega + \rd \rd^c h,
    \end{align}
    where $\rd^c= \frac{\sqrt{-1}}{2} (\bar{\partial} - \partial)$.
\end{prop}

As in the K\"{a}hler case, the basic $2$-form $\rd\eta$ can be written locally in terms of a basic potential function $h$, ie. $\rd \eta = \sqrt{-1} \partial \overline{\partial} h$, so Sasakian geometry is locally determined by a basic potential. There exists a characterization of the space of Sasakian metrics on $M$ whose Reeb vector field is $\xi$ and whose transverse holomorphic structure is $J$ as an affine space. We will not require the full description but will use the following:

\begin{definition} \label{Def:Type.II.deformation}
    Given a Sasakian structure $\cS = (\xi, \eta, \Phi, g) \in \fF(\xi, \overline{J})$, a transformation of the form $\eta \mapsto \eta^{'} = \eta + \rd^c h$ where $h$ is a basic function is an instance of a \emph{deformation of type II}. Such a transformation induces a $(1,1)$-tensor $\Phi^{'}$ and Riemannian metric $g^{'}$ by
    \begin{align*}
        &\Phi^{'} = \Phi - (\xi \otimes (\rd^c h)) \circ \Phi \\
        &g^{'} = \rd \eta^{'} \circ (1 \otimes \Phi^{'}) + \eta^{'} \otimes \eta^{'}
    \end{align*}
    The ensuing Sasakian structure $\cS' = (\xi, \eta^{'}, \Phi^{'}, g^{'})$ also lies in $\fF(\xi, \overline{J})$.
\end{definition}

\begin{remark}
    The definition of a deformation of type II is broader than what is stated above. We only make use of the specific case mentioned and refer the reader to \cites{Boyer2008, Boyer2008metric} for the broader context.
\end{remark}

\addcontentsline{toc}{section}{References}
\bibliography{Bibliografia-2020-07}
	
\end{document}